\documentclass[11pt]{amsart}
\usepackage{amsfonts}
\usepackage{amsmath,amscd}
\usepackage{amsthm}
\usepackage{amssymb,setspace}
\usepackage{latexsym}

\usepackage[normalem]{ulem}

\usepackage{hyperref}

\swapnumbers
\newtheorem{thm}[subsection]{Theorem}
\newtheorem{subthm}[subsubsection]{Theorem}

\newtheorem{sublem}[subsubsection]{Lemma}
\newtheorem{Prop}[subsection]{Proposition}
\newtheorem{prop}[subsubsection]{Proposition}
\newtheorem{subprop}[subsubsection]{Proposition}
\newtheorem{cor}[subsection]{Corollary}

\newtheorem{ass}[subsection]{Assumption}

\theoremstyle{definition}
\newtheorem{defn}[subsection]{Definition}
\newtheorem{subdefn}[subsubsection]{Definition}
\newtheorem{rem}[subsection]{Remark}
\newtheorem{subrem}[subsubsection]{Remark}

 \usepackage[curve,matrix,arrow,color]{xy}
 
\usepackage[usenames,dvipsnames]{color}

\numberwithin{equation}{section}
\newcommand{\beqa}{\begin{eqnarray}}
\newcommand{\eeqa}{\end{eqnarray}}

\newcommand{\nc}{\newcommand}
\newcommand{\rnc}{\renewcommand}

\nc{\cal}{\mathcal}

\nc{\goth}{\mathfrak}
\rnc{\bold}{\mathbf}
\renewcommand{\frak}{\mathfrak}
\renewcommand{\Bbb}{\mathbb}

\newcommand{\one}{\mathbf{1}}

\nc\K{\mathbb K}
\nc{\Cal}{\mathcal}
\nc{\Xp}[1]{X^+(#1)}
\nc{\Xm}[1]{X^-(#1)}
\nc{\on}{\operatorname}
\nc{\ch}{\mbox{ch}}
\nc{\Z}{{\bold Z}}
\nc{\J}{{\mathcal J}}
\nc{\C}{{\bold C}}
\nc{\Q}{{\bold Q}}

\nc{\N}{{\Bbb N}}
\nc\beq{\begin{equation}}
\nc\enq{\end{equation}}
\nc\lan{\langle}
\nc\ran{\rangle}
\nc\bsl{\backslash}
\nc\mto{\mapsto}
\nc\lra{\leftrightarrow}
\nc\hra{\hookrightarrow}
\nc\sm{\smallmatrix}
\nc\esm{\endsmallmatrix}
\nc\sub{\subset}
\nc\ti{\tilde}
\nc\nl{\newline}
\nc\fra{\frac}
\nc\und{\underline}
\nc\ov{\overline}
\nc\ot{\otimes}
\nc\bbq{\bar{\bq}_l}
\nc\bcc{\thickfracwithdelims[]\thickness0}
\nc\ad{\text{\rm ad}}
\nc\Ad{\text{\rm Ad}}
\nc\Hom{\text{\rm Hom}}
\nc\End{\text{\rm End}}
\nc\Ind{\text{\rm Ind}}
\nc\Res{\text{\rm Res}}
\nc\Ker{\text{\rm Ker}}
\rnc\Im{\text{Im}}
\nc\sgn{\text{\rm sgn}}
\nc\tr{\text{\rm tr}}
\nc\Tr{\text{\rm Tr}}
\nc\supp{\text{\rm supp}}
\nc\card{\text{\rm card}}
\nc\bst{{}^\bigstar\!}
\nc\he{\heartsuit}
\nc\clu{\clubsuit}
\nc\spa{\spadesuit}
\nc\di{\diamond}
\nc\cW{\cal W}
\nc\cG{\cal G}
\nc\al{\alpha}
\nc\bet{\beta}
\nc\ga{\gamma}
\nc\de{\delta}
\nc\ep{\epsilon}
\nc\io{\iota}
\nc\om{\omega}
\nc\si{\sigma}
\rnc\th{\theta}
\nc\ka{\kappa}
\nc\la{\lambda}
\nc\ze{\zeta}

\nc\vp{\varpi}
\nc\vt{\vartheta}
\nc\vr{\varrho}

\nc\Ga{\Gamma}
\nc\De{\Delta}
\nc\Om{\Omega}
\nc\Si{\Sigma}
\nc\Th{\Theta}
\nc\La{\Lambda}

\nc\bepsilon{\overline{\epsilon}}
\nc\bak{\overline{k}}
\nc\bgamma{\overline{\gamma}}

\nc\boa{\bold a}
\nc\bob{\bold b}
\nc\boc{\bold c}
\nc\bod{\bold d}
\nc\boe{\bold e}
\nc\bof{\bold f}
\nc\bog{\bold g}
\nc\boh{\bold h}
\nc\boi{\bold i}
\nc\boj{\bold j}
\nc\bok{\bold k}
\nc\bol{\bold l}
\nc\bom{\bold m}
\nc\bon{\bold n}
\nc\boo{\bold o}
\nc\bop{\bold p}
\nc\boq{\bold q}
\nc\bor{\bold r}
\nc\bos{\bold s}
\nc\bou{\bold u}
\nc\bov{\bold v}
\nc\bow{\bold w}
\nc\boz{\bold z}

\nc\ba{\bold A}
\nc\bb{\bold B}
\nc\bc{\bold C}
\nc\bd{\bold D}
\nc\be{\bold E}
\nc\bg{\bold G}
\nc\bh{\bold H}
\nc\bi{\bold I}
\nc\bj{\bold J}
\nc\bk{\bold K}
\nc\bl{\bold L}
\nc\bm{\bold M}

\nc\bn{\bold N}
\nc\bo{\bold O}
\nc\bp{\bold P}
\nc\bq{\bold Q}
\nc\br{\bold R}
\nc\bs{\bold S}
\nc\bt{\bold T}
\nc\bu{\bold U}
\nc\bv{\bold V}
\nc\bw{\bold W}
\nc\bz{\bold Z}
\nc\bx{\bold X}

\nc\ca{\mathcal A}
\nc\cb{\mathcal B}
\nc\cc{\mathcal C}
\nc\cd{\mathcal D}
\nc\ce{\mathcal E}
\nc\cf{\mathcal F}
\nc\cg{\mathcal G}
\rnc\ch{\mathcal H}
\nc\ci{\mathcal I}
\nc\cj{\mathcal J}
\nc\ck{\mathcal K}
\nc\cl{\mathcal L}
\nc\cm{\mathcal M}
\nc\cn{\mathcal N}
\nc\co{\mathcal O}
\nc\cp{\mathcal P}
\nc\cq{\mathcal Q}
\nc\car{\mathcal R}
\nc\cs{\mathcal S}
\nc\ct{\mathcal T}
\nc\cu{\mathcal U}
\nc\cv{\mathcal V}
\nc\cz{\mathcal Z}
\nc\cx{\mathcal X}
\nc\cy{\mathcal Y}

\nc\e[1]{E_{#1}}
\nc\ei[1]{E_{\delta - \alpha_{#1}}}
\nc\esi[1]{E_{s \delta - \alpha_{#1}}}
\nc\eri[1]{E_{r \delta - \alpha_{#1}}}
\nc\ed[2][]{E_{#1 \delta,#2}}
\nc\ekd[1]{E_{k \delta,#1}}
\nc\emd[1]{E_{m \delta,#1}}
\nc\erd[1]{E_{r \delta,#1}}

\nc\ef[1]{F_{#1}}
\nc\efi[1]{F_{\delta - \alpha_{#1}}}
\nc\efsi[1]{F_{s \delta - \alpha_{#1}}}
\nc\efri[1]{F_{r \delta - \alpha_{#1}}}
\nc\efd[2][]{F_{#1 \delta,#2}}
\nc\efkd[1]{F_{k \delta,#1}}
\nc\efmd[1]{F_{m \delta,#1}}
\nc\efrd[1]{F_{r \delta,#1}}

\nc\fa{\frak a}
\nc\fb{\frak b}
\nc\fc{\frak c}
\nc\fd{\frak d}
\nc\fe{\frak e}
\nc\ff{\frak f}
\nc\fg{\frak g}
\nc\fh{\frak h}
\nc\fj{\frak j}
\nc\fk{\frak k}
\nc\fl{\frak l}
\nc\fm{\frak m}
\nc\fn{\frak n}
\nc\fo{\frak o}
\nc\fp{\frak p}
\nc\fq{\frak q}
\nc\fr{\frak r}
\nc\fs{\frak s}
\nc\ft{\frak t}
\nc\fu{\frak u}
\nc\fv{\frak v}
\nc\fz{\frak z}
\nc\fx{\frak x}
\nc\fy{\frak y}

\nc\fA{\frak A}
\nc\fB{\frak B}
\nc\fC{\frak C}
\nc\fD{\frak D}
\nc\fE{\frak E}
\nc\fF{\frak F}
\nc\fG{\frak G}
\nc\fH{\frak H}
\nc\fJ{\frak J}
\nc\fK{\frak K}
\nc\fL{\frak L}
\nc\fM{\frak M}
\nc\fN{\frak N}
\nc\fO{\frak O}
\nc\fP{\frak P}
\nc\fQ{\frak Q}
\nc\fR{\frak R}
\nc\fS{\frak S}
\nc\fT{\frak T}
\nc\fU{\frak U}
\nc\fV{\frak V}
\nc\fZ{\frak Z}
\nc\fX{\frak X}
\nc\fY{\frak Y}
\nc\tfi{\ti{\Phi}}
\nc\bF{\bold F}
\rnc\bol{\bold 1}

\nc\ua{\bold U_\A}

\newcommand{\oC}{\mathbb{C}}
\nc\qinti[1]{[#1]_i}
\nc\q[1]{[#1]_q}
\nc\xpm[2]{E_{#2 \delta \pm \alpha_#1}}  
\nc\xmp[2]{E_{#2 \delta \mp \alpha_#1}}
\nc\xp[2]{E_{#2 \delta + \alpha_{#1}}}
\nc\xm[2]{E_{#2 \delta - \alpha_{#1}}}
\nc\hik{\ed{k}{i}}
\nc\hjl{\ed{l}{j}}
\nc\qcoeff[3]{\left[ \begin{smallmatrix} {#1}& \\ {#2}& \end{smallmatrix}
\negthickspace \right]_{#3}}
\nc\qi{q}
\nc\qj{q}

\newcommand{\ffrac}[2]{\mbox{\footnotesize$\displaystyle\frac{#1}{#2}$}}

\newcommand{\rmi}{\mathrm{i}}

\newcommand{\dia}{d}   

\nc\ufdm{{_\ca\bu}_{\rm fd}^{\le 0}}

\newcommand{\qO}{O_q(\widehat{sl_2})}

\nc{\NtW}{\textsf{W}^{\footnotesize[N]}}
\nc{\NtcWi}{\cal{W}_i^{\footnotesize[N]}}
\nc{\NtcWipun}{\cal{W}_{i+1}^{\footnotesize[N]}}

\nc{\NtcWzero}{\cal{W}_0^{\footnotesize[N]}}
\nc{\NtcWun}{\cal{W}_1^{\footnotesize[N]}}

\nc{\NA}{\textsf{A}^{\footnotesize[N]}}
\nc{\NAs}{\textsf{A}^{*\footnotesize[N]}}

\nc{\AN}{A^{\footnotesize[N]}}
\nc{\ANs}{{A}^{*\footnotesize[N]}}

\nc{\NWex}[1]{{\cW}^{\footnotesize[#1]}}

\nc{\Aex}[1]{{A}^{\footnotesize[#1]}}
\nc{\Aexs}[1]{{A^*}^{\footnotesize[#1]}}

\nc\isom{\cong} 

\nc{\pone}{{\Bbb C}{\Bbb P}^1}
\nc{\pa}{\partial}

\nc{\F}{{\mathcal F}}
\nc{\Sym}{{\goth S}}
\nc{\A}{{\mathcal A}}
\nc{\arr}{\rightarrow}
\nc{\larr}{\longrightarrow}

\nc{\ri}{\rangle}
\nc{\lef}{\langle}
\nc{\W}{{\mathcal W}}
\nc{\uqatwoatone}{{U_{q,1}}(\su)}
\nc{\uqtwo}{U_q(\goth{sl}_2)}
\nc{\dij}{\delta_{ij}}
\nc{\divei}{E_{\alpha_i}^{(n)}}
\nc{\divfi}{F_{\alpha_i}^{(n)}}
\nc{\Lzero}{\Lambda_0}
\nc{\Lone}{\Lambda_1}
\nc{\ve}{\varepsilon}
\nc{\phioneminusi}{\Phi^{(1-i,i)}}
\nc{\phioneminusistar}{\Phi^{* (1-i,i)}}
\nc{\phii}{\Phi^{(i,1-i)}}
\nc{\Li}{\Lambda_i}
\nc{\Loneminusi}{\Lambda_{1-i}}
\nc{\vtimesz}{v_\ve \otimes z^m}

\nc{\asltwo}{\widehat{\goth{sl}_2}}
\nc\ag{\widehat{\goth{g}}}  
\nc\teb{\tilde E_\boc}
\nc\tebp{\tilde E_{\boc'}}

\newcommand{\eeq}{\end{equation}}
\newcommand{\ben}{\begin{eqnarray}}
\newcommand{\een}{\end{eqnarray}}

\begin{document}
\begin{flushright}
ZMP-HH/16-10\\
Hamburger Beitr\"age zur Mathematik 595
\end{flushright}
\vskip 5em

\title[Cyclic TD pairs and Onsager algebras]
{Cyclic tridiagonal pairs, higher order Onsager algebras\\ and orthogonal polynomials}

\author{P. Baseilhac$^{*,\diamond}$}
\author{A.M. Gainutdinov$^{*,\dagger}$}
\author{T.T. Vu$^*$}

\address{$^{*}$ Laboratoire de Math\'ematiques et Physique Th\'eorique CNRS/UMR 7350,
 F\'ed\'eration Denis Poisson FR2964,
Universit\'e de Tours,
Parc de Grammont, 37200 Tours, 
France}

\email{baseilha@lmpt.univ-tours.fr; \newline\mbox{}\qquad\qquad\qquad\qquad\  azat.gainutdinov@lmpt.univ-tours.fr; \newline\mbox{}\qquad\qquad\qquad\qquad\  thi-thao.vu@lmpt.univ-tours.fr}

\address{$^{\dagger}$ DESY, Theory Group, Notkestrasse 85, Bldg. 2a, 22603 Hamburg and
 Fachbereich Mathematik, Universit\"at Hamburg, Bundesstra\ss e 55,
20146 Hamburg, Germany}

\address{$^{\diamond}$ Centre de recherches math\'ematiques Universit\'e de Montr\'eal, CNRS/UMI 3457, P.O. Box 6128, Centre-ville Station, Montr\'eal (Qu\'ebec), H3C 3J7 CANADA}

\begin{abstract} The concept of cyclic tridiagonal pairs is introduced, and explicit examples are given. For a fairly general  class of cyclic tridiagonal pairs
with cyclicity~$N$, we associate a pair of `divided polynomials'. The properties of this pair generalize the ones of tridiagonal pairs 
of Racah type. The algebra generated by the pair of divided polynomials is identified as a higher-order generalization of the Onsager algebra. It can be viewed as a subalgebra of the $q-$Onsager algebra for a proper specialization at $q$ the primitive $2N$th root of unity. Orthogonal polynomials beyond the Leonard duality are revisited in light of this framework. In particular, certain second-order Dunkl shift operators provide a realization of the divided polynomials at $N=2$ or $q=\rmi$.
\end{abstract}

%
%

\maketitle





\setcounter{tocdepth}{1}
\tableofcontents


\section{Introduction}
Let ${\mathbb K}$ denote a field. Let $V$ denote a vector space over ${\mathbb K}$ with finite positive dimension. 
Recall that a {\it Leonard pair} is a pair of linear transformations $A,A^*$ such that there exist two bases for $V$ with respect to
which the matrix representing $A$ (resp. $A^*$) is irreducible tridiagonal (resp. diagonal) and
 the matrix representing $A^*$ (resp. $A$) is diagonal (resp. irreducible tridiagonal) \cite{Ter03}. For the well-known families of one-variable orthogonal polynomials 
in the Askey-scheme including the Bannai--Ito polynomials \cite{BI}, the bispectral problem they solve finds a natural interpretation within the framework of Leonard's theorem \cite{BI} and Leonard pairs \cite{Ter2}:
given a Leonard pair, the overlap coefficients
between the two bases coincide with polynomials of the Askey-scheme ($q-$Racah, $q-$Hahn, $q-$Krawtchouk, etc.)  or their specializations
evaluated on a discrete support. The three-term recurrence 
relation, three-term difference equation and orthogonality property of the orthogonal polynomials follow from the corresponding 
Leonard pair \cite{Ter2}. In the literature, examples of such correspondence first appeared in \cite{BI} and in the context of mathematical physics, 
see for instance \cite{Zhed,GLZ}.\vspace{1mm}

The concept of {\it tridiagonal pair} (see Definition \ref{tdp}) introduced by Ito--Tanabe--Terwilliger arised as
a natural extension of the concept of Leonard pair \cite{Ter93,Ter01}. For a tridiagonal pair $A,A^*$, in the basis in which $A$ (resp. $A^*$) is diagonal with multiplicities in the eigenvalues, then $A^*$ (resp.~$A$) is a block tridiagonal matrix. So, the main difference between a tridiagonal pair and a Leonard pair is the dimension of the eigenspaces of $A$, $A^*$ that  are not necessarely one-dimensional. In view of the connection between Leonard pairs and orthogonal  polynomials of the Askey-scheme, 
 given a tridiagonal pair a correspondence with certain multivariable generalizations of orthogonal polynomials of the Askey-scheme was expected. This has been recently exhibited \cite{BM}: the so-called Gasper-Rahman polynomials \cite{GR1} and their classical analogues discovered by Tratnik \cite{Trat} evaluated on a multidimensional discrete support are interpreted as the overlap coefficients between two distinct eigenbases of $A$ and $A^*$. In particular, the bispectrality of these polynomials was shown in \cite{GI,Iliev}: they satisfy a system of coupled three-term recurrence relations and three-term difference relations.\vspace{1mm}

From a more general point of view, Leonard and tridiagonal pairs arise in the context of 
representation theory,
from irreducible finite dimensional modules of a certain class of classical and quantum algebras with two generators
 ${\textsf A},{\textsf A}^*$ and unit, called tridiagonal algebras \cite{Ter03} and the $q-$Onsager algebra \cite{B4}.  These algebras 
are defined through a pair of relations, see Definition \ref{TDgendef}, which can be viewed as deformations 
of the $q-$Serre relations of $U_q(\widehat{sl_2})$ or deformations of the Dolan-Grady 
relations \cite{DG}. Note that infinite dimensional modules have also been constructed, see \cite{BB3} or \cite{BM}.
In \cite{BM}, ${\textsf A},{\textsf A}^*$ are mapped  to certain multivariable $q-$difference/difference operators introduced by 
Geronimo and Iliev \cite{GI,Iliev}. \vspace{1mm}

Having in mind a specialization of the tridiagonal and $q-$Onsager algebras relations at $q$ a root of unity, in the present paper we introduce a natural generalization of the tridiagonal pairs that we call a `cyclic tridiagonal pair'. Let $N$ be a positive integer and $\mathbb Z_N$ denotes the cyclic group $\mathbb Z/N\mathbb Z$ of order~$N$.
\begin{defn}\label{deficitri}
By a  {\it cyclic tridiagonal pair} on $V$, we mean an ordered pair of linear transformations $C:V \to V$ and 
$C^*:V \to V$ that satisfy the following four conditions for some positive integer  $N$.
\begin{itemize}
\item[(i)] Each of $C,C^*$ is diagonalizable.
\item[(ii)] There exists an ordering $\lbrace V_p\rbrace_{p\in {\mathbb Z}_N}$ of the  
eigenspaces of $C$ such that 
\begin{equation}
C^* V_p \subseteq V_{p-1} + V_p+ V_{p+1} 
.\label{eq:ct1}
\end{equation}
\item[(iii)] There exists an ordering $\lbrace V^*_p\rbrace_{p\in {\mathbb Z}_N}$ of
the eigenspaces of $C^*$ such that 
\begin{equation}
C V^*_p \subseteq V^*_{p-1} + V^*_p+ V^*_{p+1} 
.\label{eq:ct2}
\end{equation}
\item[(iv)] There does not exist a subspace $W$ of $V$ such  that $CW\subseteq W$,
$C^*W\subseteq W$, $W\not=0$, $W\not=V$.
\end{itemize}
We call the number $N$  \textit{the cyclicity} of the cyclic tridiagonal pair.
\end{defn}

The purpose of this paper is to study a general class of cyclic tridiagonal pairs of cyclicity~$N$. The properties of cyclic tridiagonal pairs $C, C^*$ (they satisfy certain polynomial equations) lead us
to the introduction of a new pair of operators that we call the `divided polynomials' (associated to the pair $C, C^*$). As an application of these new concepts, we propose a systematic approach to the construction of higher-order generalizations of tridiagonal pairs and of the tridiagonal algebras. Namely, we derive explicit pairs of square matrices such that there exist two bases with respect to which the first matrix (resp. the second) is irreducible block $(2N+1)-$diagonal (resp. diagonal) and
 the second matrix  (resp. the first) is diagonal  (resp. irreducible block $(2N+1)-$diagonal), and any such given pair generates an algebra that generalizes the tridiagonal algebra. The case $N=1$ specializes to tridiagonal pairs and the corresponding algebra is a tridiagonal algebra. \vspace{1mm}
 
 There are numerous motivations for the higher-order generalizations: for instance, the problem of polynomial solutions to 
 higher-order recurrence relations \cite{DVA} (see also \cite{Krein} and Chapter VII of \cite{B68}) that occur in the context of matrix valued orthogonal polynomials, sieved polynomials \cite{AA,CIM}, orthogonal polynomials beyond Leonard duality \cite{GVZ,TVZ} that are discussed in Section 5, 
and the bispectral problem for multidiagonal infinite, semi-infinite or finite dimensional matrices \cite{GIK}. Another
motivation comes from representation theory of the $q-$Onsager algebra and the non-perturbative analysis of quantum integrable systems associated with quantum groups or related coideal subalgebras for a deformation parameter $q$ specialized at a root of unity \cite{DFM,V}.
\vspace{1mm}

We now describe the main results of this paper.
The first main result is the following: under some reasonable assumptions (see Assumption~\ref{ass} and the assumption in Corollary~\ref{cor}), to a tridiagonal pair of $q-$Racah type (see Case I after Definition \ref{tdp}) 
 we associate a cyclic tridiagonal pair (by specializing $q$ to  the root of unity $e^{\rmi \pi/N}$)
and a pair of divided polynomials. For a large class of tridiagonal pairs of $q-$Racah type, the relations satisfied by the two pairs are identified (in Proposition \ref{prop-qOA}, Theorems~\ref{prop:mixed-rel-k-zero} and~\ref{propNOns}). In particular,  the divided polynomials satisfy the defining relations of a higher-order generalization of the Onsager algebra, called $N-$Onsager algebra and denoted $OA^{[N]}$ (see Definition \ref{defOAN}), and act on the vector space as higher-order generalizations of tridiagonal pairs.\vspace{1mm}

The second main result is about orthogonal polynomials outside of the Askey scheme (like the complementary Bannai--Ito polynomials): starting from a special class of tridiagonal pairs $A$, $A^*$ of $q-$Racah type, 
under the specialization $q=e^{\rmi \pi/2}$ the operator $A^*$ and the divided polynomial associated to $A$ generate another (non-symmetric) higher-order generalization of the tridiagonal/Onsager algebra. The relations are
given by (\ref{degrel}). Importantly, a homomorphism from this algebra to the complementary Bannai--Ito algebra (\ref{CBIalg}), introduced in \cite{GVZ} is exhibited, see (\ref{mapkappa}). According to the relation between irreducible  finite dimensional representations of (\ref{CBIalg}) and the complementary Bannai--Ito polynomials, we conjecture that  (spectrum-degenerate representations of) our new algebra provides  multivariable generalizations of the complementary Bannai--Ito polynomials and their descendants, besides the one-variable case \cite{GVZ,TVZ}. By construction, such multivariable orthogonal polynomials should satisfy a system of coupled three-term recurrence relations and {\it five-term} difference equations. 
\vspace{1mm}

The paper is organized as follows. In Section 2, we review the concept of tridiagonal pairs introduced in \cite{Ter93}.
The $q-$Dolan-Grady relations satisfied by a class of tridiagonal pairs as well as their higher-order generalizations are  also recalled. 
The concept of a cyclic tridiagonal pair is introduced in Section~3,
 with an example of a cyclic tridiagonal pair in a $4-$dimensional vector space reported in Appendix A.1. To any tridiagonal pair of $q-$Racah type and having certain assumptions,  we also construct  a cyclic tridiagonal pair. An example is reported in Appendix A.2. In Section~\ref{sec:div-poly}, using the pair of minimal polynomials
 associated with the cyclic tridiagonal pair, we construct a second pair of elements that we call the `divided polynomials'.
We describe then the action of the two pairs with respect to the distinct  eigenbases. In Section~4, the  relations satisfied by the cyclic tridiagonal pair and the divided polynomials are identified. In particular, the subalgebra generated by the divided polynomials is (a representation of)
the algebra $OA^{[N]}$ with the defining relations  (\ref{genOAN}).
In Section~5, we consider two operators built from a particular tridiagonal pair of $q-$Racah type $A,A^*$. They are obtained as  $A^*$ and the divided polynomial associated with $A$ under the specialization $q=\rmi$. The algebra (\ref{degrel}) generated by these two operators can be viewed as a non-symmetric generalization of the tridiagonal algebra. We then describe a homomorphism of this algebra onto the complementary Bannai--Ito algebra (\ref{CBIalg}) and onto its degenerate version, the dual $-1$ Hahn algebra (\ref{dualHahn-alg}). In particular, a representation of the algebra (\ref{degrel}) in terms of second-order Dunkl shift operators is given in Proposition~\ref{prop:BI}. 
\vspace{1mm}

Let us point out that the results here presented are independently of  interest in relation with the construction of analogues 
of the Lusztig's specialization~\cite{Luszt} of the quantum affine algebra  $U_q(\widehat{sl_2})$ at $q$ a root of unity and an analogue of Lusztig's ``small'' quantum group for the $q-$Onsager algebra\footnote{The $q-$Onsager algebra is isomorphic to a certain coideal subalgebra of $U_q(\widehat{sl_2})$ \cite{Bas3,BB3}. For the proof of injectivity of the map, see \cite{Kolb}.}.
For the reader familiar with this subject, the pair of divided polynomials here introduced 
can be viewed as an analogue of Lusztig's divided 
powers for elements in a certain coideal subalgebra of $U_q(\widehat{sl_2})$.
In this vein, we conjecture that the relations (\ref{Talg}) together with the polynomial relations~\eqref{relPt-W} and the relations (\ref{eq:mix1}), (\ref{eq:mix2}) and (\ref{genOAhom}) discussed in Section 4 are the defining relations of a specialization of the $q-$Onsager algebra at $q=e^{\frac{\rmi\pi}{N}}$. We intend to discuss this important subject elsewhere.\vspace{1mm}

Finally, let us mention that higher-rank generalizations of the  $N-$Onsager algebra $OA^{[N]}$ may be introduced as well, having in mind 
the higher-rank generalizations of the $q-$Onsager algebra \cite{BB2}
and the higher-rank generalization of the Onsager algebra \cite{UI}. However, to our knowledge
neither the analog of the concept of tridiagonal pair for higher-rank cases nor the representation
theory of the generalized $q-$Onsager algebras has been investigated yet.\vspace{1mm}

{\bf Conventions:} In this paper, $\{ x\}$ denotes the integer part of $x$.  Let $m,n$ be integers. We define  $\overline{m,n}$  for the set $\{m,m+1,...,n-1,n\}$.  We also use the $q$-numbers
$[n]_q=\frac{q^n-q^{-n}}{q-q^{-1}}$, the $q$-factorials $[n]_q!= \prod_{r=1}^n [r]_q$,  with $[0]_q!=1$,
and the $q$-commutators $[X,Y]_q=qXY-q^{-1}YX$, and $[X,Y]=[X,Y]_1$ is the ordinary commutator.

\vspace{2mm}

\section{Tridiagonal pairs and relations}
 In this Section, we first recall  the concept of tridiagonal pairs developed in \cite{Ter93,Ter03,Ter01,NT:muqrac}. Given a tridiagonal pair $A,A^*$, we recall the basic and higher-order relations satisfied by $A,A^*$. As an example, we consider a tridiagonal pair associated with a reduced parameter sequence and indeterminate $q$.  In this case, the pair satisfies the so-called $q-$Dolan-Grady relations and their higher-order generalizations.
\vspace{1mm}

\subsection{Definitions and the classification}
Let ${\mathbb K}$ be an arbitrary field unless otherwise noted. Throughout the paper, we always denote by $V$  a finite-dimensional vector space over $\K$.

\begin{subdefn}[See \cite{Ter01}, Definition 2.1]
\label{tdp}
\rm
\textit{A  tridiagonal pair} (or $TD$ pair)
on $V$ is an ordered pair of linear transformations\footnote{According to a common 
notational convention, for a linear transformation $A$
the conjugate-transpose of $A$ is denoted~$A^*$. This convention is not used here.}
$A:V \to V$ and $A^*:V \to V$ 
that satisfy the following four conditions.
\begin{itemize}
\item[(i)] Each of $A,A^*$ is diagonalizable on $V$.
\item[(ii)] There exists an ordering $\lbrace V_p\rbrace_{p=0}^d$ of the  
eigenspaces of $A$ such that 
\begin{equation}
A^* V_p \subseteq V_{p-1} + V_p+ V_{p+1} \qquad \qquad 0 \leq p \leq d,
\label{eq:t1}
\end{equation}
where $V_{-1} = 0$ and $V_{d+1}= 0$.
\item[(iii)] There exists an ordering $\lbrace V^*_p\rbrace_{p=0}^{\delta}$ of
the  
eigenspaces of $A^*$ such that 
\begin{equation}
A V^*_p \subseteq V^*_{p-1} + V^*_p+ V^*_{p+1} 
\qquad \qquad 0 \leq p \leq \delta,
\label{eq:t2}
\end{equation}
where $V^*_{-1} = 0$ and $V^*_{\delta+1}= 0$.
\item[(iv)] There does not exist a subspace $W$ of $V$ such  that $AW\subseteq W$,
$A^*W\subseteq W$, $W\not=0$, $W\not= V$.
\end{itemize}
\end{subdefn}
\begin{subrem}\label{tdp-rem}
In Definition \ref{tdp}, the scalars $d$ and $\delta$ are necessarily  equal \cite[Lemma~4.5]{Ter01}.
\end{subrem}

Let $A,A^*$ denote a TD pair on $V$, as in Definition 
\ref{tdp}. By \cite[Lemma~4.5]{Ter01} the integer $d=\delta$ from
(ii), (iii) is called \textit{the diameter} of the
pair. An ordering of the eigenspaces of $A$ (resp. $A^*$)
is said to be {\em standard} whenever it satisfies  (\ref{eq:t1})
 (resp. (\ref{eq:t2})).  Let $\{V_p\}_{p=0}^d$ (resp.
$\{V^*_p\}_{p=0}^d$) denote a standard ordering of the eigenspaces
 of $A$ (resp. $A^*$). For $0 \leq p \leq d$, let 
$\theta_p$  (resp. $\theta^*_p$) denote the eigenvalue of
$A$  (resp.  $A^*$) associated with $V_p$  (resp. $V^*_p$). According to the classification~\cite[Theorem~4.4]{Ter03}, the eigenvalues $\theta_p$ and $\theta^*_p$, $p=0,1,...,d$, can take three different forms:
\vspace{1mm}

$\bullet$ Case I: the $q-$Racah type (with $b,c,b^*,c^*\not=0$)
\begin{eqnarray}\label{specgen-I}
\begin{split}
\theta_p &= a + b q^{2p} + c q^{-2p}, 
\\
\theta^*_p &= a^* + b^* q^{2p} + c^* q^{-2p};
\end{split}
\end{eqnarray}

$\bullet$ Case II: the Racah type\footnote{If the characteristic of ${\mathbb K}$ is 2, one identifies $p(p-1)/2$ as $0$ if $p=0$ or $1$ (mod $4$), and as $1$ if $p=2$ or $p=3$ (mod $4$).} (with $c,c^*\not=0$)
\beqa
\begin{split}
\qquad\theta_p &= a + bp + cp(p-1)/2,  \nonumber\\
\theta^*_p &= a^* + b^*p + c^*p(p-1)/2;  \nonumber
\end{split}
\eeqa

$\bullet$ Case III (the characteristic of ${\mathbb K}$ is not 2)
\begin{eqnarray}
&&\theta_p = a + b(-1)^p  + cp (-1)^p,\nonumber
\\
&&\theta^*_p = a^* +  b^*(-1)^p  + c^* p (-1)^p\nonumber.
\end{eqnarray}

Here  $a,a^*, b,\;b^*,\ c ,\;c^*$ are scalars and $q\not=0, \ q^2 \not=1,
\ q^2 \not=-1$. In this case,  we say that $A,A^*$ is  a tridiagonal pair of $q-$Racah (case I) or Racah (case II) type \cite[Theorem~5.3]{NT:muqrac}.  Note that the cases listed above are the   `generic' cases  where the  parameters are nonzero. Other examples  with some of the parameters set to zero can be found in \cite{Ter033,ITaug}. 
\vspace{1mm}

For the analysis presented below, the notion of TD system \cite{Ter01} is needed. Let ${\rm End}(V)$ denote the ${\mathbb K}$-algebra of all linear transformations from $V$ to $V$. Let $A$ denote a diagonalizable element of $\mbox{\rm End}(V)$.
Let $\{V_p\}_{p=0}^d$ (resp. $\{\theta_p\}_{p=0}^d$) denote an ordering of the eigenspaces (resp. eigenvalues) of $A$. For $0 \leq p \leq d$, define $E_p \in 
\mbox{\rm End}(V)$ 
such that 
$$
(E_p-\one)V_p=0 \qquad \text{and} \qquad E_pV_m=0 \quad \text{for} \quad m \neq p \;\; (0 \leq m \leq d)\ ,
$$
where $\one$ denotes the identity of $\mathrm{End}(V)$. We call $E_p$ the {\em primitive idempotent} of $A$ corresponding to $V_p$
(or $\theta_p$). Observe that
(i) $\one=\sum_{p=0}^d E_p$;
(ii) $E_pE_m=\delta_{p,m}E_p$ $(0 \leq p,m \leq d)$;
(iii) $V_p=E_pV$ $(0 \leq p \leq d)$;
(iv) $A=\sum_{p=0}^d \theta_p E_p$.
Let $A,A^*$ denote a TD pair on $V$.
An ordering of the primitive idempotents 
 of $A$ (resp. $A^*$)
is said to be {\em standard} whenever
the corresponding ordering of the eigenspaces of $A$ (resp. $A^*$)
is standard.
\begin{subdefn}
{\rm \cite[Definition~2.1]{Ter01}}
 \label{def:TDsystem} 
\rm
A {\it tridiagonal system} (or {\it  $TD$ system}) on $V$ is a sequence
\[
 \Phi=(A;\{E_p\}_{p=0}^d;A^*;\{E^*_p\}_{p=0}^d)
\]
that satisfies:
\begin{itemize}
\item[(i)]
$A,A^*$ is a TD pair on $V$,
\item[(ii)]
$\{E_p\}_{p=0}^d$ is a standard ordering
of the primitive idempotents of $A$,
\item[(iii)]
$\{E^*_p\}_{p=0}^d$ is a standard ordering
of the primitive idempotents of $A^*$.
\end{itemize}
\end{subdefn}

\begin{sublem}{\rm \cite[Lemma~2.2]{INT}}
\label{lem:triplep}
Let $(A; \lbrace E_p\rbrace_{p=0}^d; A^*; \lbrace E^*_p\rbrace_{p=0}^d)$
denote a TD system. Then the following holds for $0 \leq p,m,r\leq d$:
\begin{equation*}
\begin{split}
&\;\mathrm{(i)}\qquad\; E^*_pA^rE^*_m=0 \\
&\mathrm{(ii)}\qquad E_pA^{*r}E_m=0 
\end{split}
\qquad \text{if} \quad |p-m|>r\ . \qquad\qquad\qquad\qquad\qquad\qquad\qquad\qquad\qquad\qquad\qquad
\end{equation*}
\end{sublem}


\subsection{Tridiagonal algebras}
The theory of tridiagonal pairs is closely related with the representation theory of tridiagonal algebras. Tridiagonal algebras are defined as follows:
\begin{subdefn}[See \cite{Ter03}, Definition 3.9]\label{TDgendef}
Let  $\beta, \gamma, \gamma^*, \rho, \rho^* $ denote scalars in $\K$.
\textit{The tridiagonal algebra} (or TD algebra)  $T=T(\beta, \gamma, \gamma^*,
\rho, \rho^*)$ is the associative $\K$-algebra with unit  generated by two elements ${\textsf A}$, ${\textsf A}^*$
subject to the relations 
\beqa
\lbrack {\textsf A},{\textsf A}^2{\textsf A}^*-\beta {\textsf A}{\textsf A}^*{\textsf A} + {\textsf A}^*{\textsf A}^2 -\gamma ({\textsf A}{\textsf A}^*+{\textsf A}^*{\textsf A})-\rho {\textsf A}^*\rbrack 
&=&0,
\label{TD1}
\\
\lbrack {\textsf A}^*,{\textsf A}^{*2}{\textsf A}-\beta {\textsf A}^*{\textsf A}{\textsf A}^* + {\textsf A}{\textsf A}^{*2} -
\gamma^* ({\textsf A}^*{\textsf A}+{\textsf A}{\textsf A}^*)-\rho^* {\textsf A} \rbrack
&=&0.
\label{TD2}
\eeqa
${\textsf A}$ and ${\textsf A}^*$ are called the standard generators of $T$. The relations (\ref{TD1}) and (\ref{TD2}) are called the tridiagonal (TD) relations.
\end{subdefn}

Every TD pair comes from an irreducible  finite dimensional module of the tridiagonal algebra. The explicit connection between the representation theory of tridiagonal algebras and the theory of tridiagonal pairs, as defined in Definition \ref{tdp}, is established by the following theorem:
\begin{subthm}[See \cite{Ter01}, Theorem 10.1]\label{thmTDqOA}
Let  $A, A^*$ denote a TD pair
over $\K$.
Then
there exists a sequence of  
 scalars
$\beta, \gamma, \gamma^*, \rho, \rho^* $ taken from $\K$
such that $A,A^*$ satisfy the tridiagonal relations (\ref{TD1}) and~(\ref{TD2}). 
The parameter sequence $\beta, \gamma, \gamma^*, \rho, \rho^* $ is uniquely determined by the pair if the diameter is at least~$3$.
\end{subthm}

\subsection{Higher-order relations}
Given a TD pair $A$, $A^*$ it is possible to show that, in addition to the two basic relations (\ref{TD1}), (\ref{TD2}), $A$ and $A^*$ satisfy an infinite family of pairs of polynomial relations of higher degree. Indeed, recall that $\{\theta_i\}_{i=0}^d,\{\theta^*_i\}_{i=0}^d$ denote the eigenvalues of $A$ and $A^*$, respectively. 
\begin{sublem}{\cite{BV}}\label{lem:polypart} For each positive integer $s$, there exist scalars $\beta_s,\gamma_s,{\gamma^*}_s,\delta_s,\delta^*_s$ in $\K$ such that 
\beqa
\begin{split}
\theta_i^2 - \beta_s\theta_i\theta_j  + \theta_j^2 - \gamma_s (\theta_i+\theta_j) -\delta_s&=0 \ ,\\
{\theta_i^*}^2 - \beta_s{\theta_i^*}{\theta_j^*}  + {\theta_j^*}^2 - \gamma^*_s (\theta_i^*+\theta_j^*) -\delta^*_s &=0\ ,
\end{split}
 \qquad \text{for} \quad 0\leq i,j\leq d \quad \text{with} \quad|i-j|=s\nonumber.
\eeqa
\end{sublem}
\begin{subdefn}{\cite{BV}}
Let $x,y$ denote commuting indeterminates. For each positive integer $r$, we define the polynomials $p_r(x,y)$ and $p^*_r(x,y)$ as follows:
\beqa
p_r(x,y) = (x-y)\prod_{s=1}^{r} (x^2-\beta_s xy +y^2 - \gamma_s (x+y) -\delta_s)\ ,\\
p^*_r(x,y) = (x-y)\prod_{s=1}^{r} (x^2-\beta_s xy +y^2 - \gamma^*_s (x+y) -\delta^*_s)\ .
\eeqa
Observe $p_r(x,y)$ and $p^*_r(x,y)$ have a total degree $2r+1$ in $x,y$.
\end{subdefn}

As a consequence of Lemma \ref{lem:polypart}, observe that $p_r(\theta_i,\theta_j)=0$ and $p^*_r(\theta^*_i,\theta^*_j)=0$ if $|i-j|\leq r$, for $0\leq i,j\leq d$. Combining these facts together with Lemma \ref{lem:triplep}, higher-order polynomial relations generalizing (\ref{TD1}), (\ref{TD2})  that are satisfied by a TD pair $A,A^*$ can be derived. Namely:
\begin{subthm}{\cite{BV}}
\label{hqdg} 
 Let  $A, A^*$ denote a TD pair
over $\K$.  For each positive integer $r$,
\beqa
\sum_{i,j=0}^{i+j\leq 2r+1} a_{ij} A^i {A^*}^rA^j =0 \ , \qquad \sum_{i,j=0}^{i+j\leq 2r+1} a^*_{ij} {A^*}^i {A}^r{A^*}^j =0 \  \label{hqdgr}
\eeqa
where the scalars $a_{ij} ,a^*_{ij}$ are defined by:
\beqa
p_r(x,y)=  \sum_{i,j=0}^{i+j\leq 2r+1} a_{ij} x^i y^j \quad  \mbox{and} \quad  p^*_r(x,y)=\sum_{i,j=0}^{i+j\leq 2r+1} a^*_{ij} x^i y^j \ .\label{polyr}
\eeqa
\end{subthm}
In the next subsection, we describe an explicit example of a TD pair and the relations it satisfies.

\subsection{The $q-$Dolan-Grady relations and generalizations}
We now consider the \textit{reduced} parameter sequence determining the tridiagonal relations  (\ref{TD1}) and (\ref{TD2}): setting $\gamma=\gamma^*=0$. In this case,  the tridiagonal pair $A$ and $A^*$ will be denoted by ${\cW}_0$ and ${\cW}_1$, respectively.
Parametrizing $\beta$ as $\beta=q^2+q^{-2}$, the two operators then satisfy the \textit{$q-$Dolan-Grady} relations,  also known as the defining relations of the \textit{$q-$Onsager algebra} \cite{Ter03,B4}:
\beqa\label{Talg}
\begin{split}
 \big[{\cW}_0,\big[{\cW}_0,\big[{\cW}_0,{\cW}_1\big]_q\big]_{q^{-1}}\big]&=\rho\big[{\cW}_0,{\cW}_1\big]\ ,\\
\big[{\cW}_1,\big[{\cW}_1,\big[{\cW}_1,{\cW}_0\big]_q\big]_{q^{-1}}\big]&=\rho^*\big[{\cW}_1,{\cW}_0\big]\ .
\end{split}
\eeqa
In terms of the parameters in the spectra (\ref{specgen-I}) of the TD pair $\cW_0,\cW_1$ corresponding to the choice $a=a^*=0$, one has the identification \cite{Ter03}:
\beqa
\rho= -bc(q^2-q^{-2})^2 \quad \mbox{and} \quad \rho^*=-b^*c^*(q^2-q^{-2})^2.\label{rhored}
\eeqa

Besides the two basic relations above, the tridiagonal pair ${\cW}_0,{\cW}_1$ also satisfies higher-order relations generalizing (\ref{Talg}), see Theorem \ref{hqdg}. Let $r$ be a positive integer. We refer to these relations as the {\it $r$-th  higher-order $q-$Dolan-Grady relations}. Explicitely, they read \cite{BV,V}\footnote{The coefficients $a_{ij}$ and $a^*_{ij}$ introduced in (\ref{polyr}) are such that: $a_{2r+1-2p-j \ j}=(-1)^{j+p}  \rho^{p}\, {c}_{j}^{[r,p]}$, $a^*_{2r+1-2p-j \ j}=(-1)^{j+p}  (\rho^*)^{p}\, {c}_{j}^{[r,p]}$ for $ p=0,...,r$ and $a_{ij},a^*_{ij}$ are zero otherwise.}:
\beqa\label{qDGfinr} 
\begin{split}
\sum_{p=0}^{r}\sum_{j=0}^{2r+1-2p}  (-1)^{j+p}  \rho^{p}\, {c}_{j}^{[r,p]}\,  \cW_0^{2 r+1-2p-j} \cW_1^{r} \cW_0^{j}&=0\ ,\\
\sum_{p=0}^{r}\sum_{j=0}^{2r+1-2p} (-1)^{j+p} (\rho^*)^{p} \, {c}_{j}^{[r,p]}\,  \cW_1^{2 r+1-2p-j} \cW_0^{r} \cW_1^{j}&=0\  ,
\end{split} 
\eeqa
where $c_{2(r-p)+1-j}^{[r,p]} =c_{j}^{[r,p]}$ and
\beqa
\qquad &&c_{j}^{[r,p]} =  \sum_{k=0}^j   \frac{(r-p)!}{(\{\frac{j-k}{2}\})!(r-p-\{\frac{j-k}{2}\})!}
    \sum_{{\cal P}}  [s_1]^2_{q^2}...[s_p]^2_{q^2}  \frac{[2s_{p+1}]_{q^2}...[2s_{p+k}]_{q^2}}{[s_{p+1}]_{q^2}...[s_{p+k}]_{q^2}}  \label{cfinr}
\eeqa
\beqa 
\mbox{with}\quad \ \left\{\begin{array}{cc}
\!\!\! \!\!\! \!\!\! \!\!\!  \!\!\! \!\!\! \!\!\! \!\!\!   \!\!\! \!\!\! \!\!\! \!\!\! j\in \overline{0,r-p}\ , \quad s_i\in\{1,2,...,r\}\ ,\\
{\cal P}: \begin{array}{cc} \ \ s_1<\dots<s_p\ ;\quad \ s_{p+1}<\dots<s_{p+k}\ ,\\
 \{s_{1},\dots,s_{p}\} \cap \{s_{p+1},\dots,s_{p+k}\}=\emptyset \end{array}
\end{array}\right.\ .
\label{eq:cjrp}
\eeqa
Explicit examples of TD pairs of $q-$Racah type that are not Leonard pairs\footnote{For a Leonard pair, by definition eigenspaces of $A,A^*$ are one-dimensional \cite{Ter01}.} have been constructed in \cite{Bas3}. In particular, a simple example of a TD pair of $q-$Racah type is considered in Appendix~\ref{app:q-Racah}.

\medskip


\section{Cyclic tridiagonal pairs and the divided polynomials}
In this Section, we first recall the 
new concept of `cyclic tridiagonal pairs' and define the notion of cyclic tridiagonal systems. Explicit examples of cyclic TD pairs
 are given in Appendix~\ref{app:example1}. Then, we show that
 under  Assumption~\ref{ass} and the assumption in Corollary~\ref{cor}, a TD pair of $q$-Racah type gives a cyclic TD pair. We also show  that each operator of the  cyclic TD pair in any finite-dimensional representation satisfies a  polynomial equation of degree $N$. Based on that, we introduce a pair of operators that we call the divided polynomials and study their action. 
 
\subsection{Conventions}\label{DC} We introduce here several assumptions that are used in the paper.
\vspace{1mm}

We start with the definition of a cyclic tridiagonal pair, recall Definition \ref{deficitri} in the Introduction. 
 We do not repeat it here for brevity and only note that we will use the notation $V_p^{(N)}$ and $V_p^{*(N)}$ instead of $V_p$ and $V_p^*$ in Definition \ref{deficitri} to avoid a confusion with the notation used for the eigenspaces of TD pairs from Definition~\ref{tdp}.
 
 \vspace{1mm}
 We also note that
 our definition of  cyclic TD pairs could be naturally generalized by introducing a different cyclicity $N^*$ instead of $N$ in the point (iii) of Definition~\ref{deficitri}, in accordance with the definition of usual TD pairs that involves two different diameters $d$ and $\delta$. But because the two diameters are necessarily equal, recall Remark~\ref{tdp-rem}, and most of our examples of cyclic TD pairs come from specializations of the TD pairs of $q-$Racah type  we fixed $N^*=N$.
We also  say that an ordering of the eigenspaces of $C$ (resp. $C^*$)
is  {\em standard} whenever it satisfies  (\ref{eq:ct1})
 (resp.~(\ref{eq:ct2})).
 \vspace{1mm}

  In Appendix~\ref{app:first-ex}, we give a simple and explicit example of a non-trivial cyclic TD pair acting on a $4-$dimensional vector space, see (\ref{exL4:W1}). This example naturally comes from the theory of cyclic representations of  the quantum algebra $U_q(sl_2)$ at $2N$-th roots of unity.\vspace{1mm}

  \begin{subrem}
  We note that each TD pair of diameter $d$ is a cyclic TD pair of cyclicity $d+1$ (with trivial contribution of $V_0$ while acting on $V_d$ by $A^*$, etc.). But the reverse is of course not true. 
  \end{subrem}

By analogy with the Definition \ref{def:TDsystem} of tridiagonal systems, the notion of {\it cyclic TD system} can be introduced.
Let $C$ denote a diagonalizable element of  $\mathrm{End}(V)$.
Let $\{V^{(N)}_p\}_{p\in {\mathbb Z}_N}$ denote an ordering of the eigenspaces of $C$
and let $\{\theta_p\}_{p\in {\mathbb Z}_N}$ denote the corresponding ordering of
the eigenvalues of $C$. For ${p\in {\mathbb Z}_N}$, define $E^{(N)}_p \in 
\mbox{\rm End}(V)$ 
such that  
\begin{equation}\label{EN-def}
(E_p^{(N)}-\one)V^{(N)}_p =0 \qquad  \text{and} \qquad E_p^{(N)}V^{(N)}_m =0\quad  \text{for} \quad m \neq p \; ({m\in {\mathbb Z}_N}).
\end{equation}
 We call $E_p^{(N)}$ the {\em primitive idempotent} of $C$ corresponding to $V^{(N)}_p$
(or ${\theta}_p$), \textit{i.e.,} we have (i) $\one=\sum_{p\in {\mathbb Z}_N} E_p^{(N)}$ and (ii) $E_p^{(N)}E_m^{(N)}=\delta_{p,m}E_p^{(N)}$, for $m,p\in {\mathbb Z}_N$, and (iii) $C=\sum_{p\in {\mathbb Z}_N} \theta_p E_p^{(N)}$. 
Let $C,C^*$ denote a cyclic TD pair on $V$, as defined in Definition~\ref{deficitri}. An ordering of the primitive idempotents 
 of $C$ (resp. $C^*$)
is said to be {\em standard} whenever
the corresponding ordering of the eigenspaces of $C$ (resp. $C^*$)
is standard.
\begin{subdefn}
 \label{def:cTDsystem} 
\rm
By a {\it cyclic tridiagonal system} (or {\it  cyclic $TD$ system}) on $V$ we mean a sequence
\[
 \Phi^{(N)}=(C;\{E^{(N)}_p \}_{p\in {\mathbb Z}_N};C^*;\{E^{*(N)}_p \}_{p\in {\mathbb Z}_N})
\]
that satisfies (i)--(iii) below.
\begin{itemize}
\item[(i)]
$C,C^*$ is a cyclic TD pair on $V$.
\item[(ii)]
$\{E^{(N)}_p\}_{p\in {\mathbb Z}_N}$ is a standard ordering
of the primitive idempotents of $C$.
\item[(iii)]
$\{E^{*(N)}_p\}_{p\in {\mathbb Z}_N}$ is a standard ordering
of the primitive idempotents of $C^*$.
\end{itemize}
\end{subdefn}

According to the Definitions \ref{deficitri} and \ref{def:cTDsystem}, the proof of the following lemma is straightforward:
\begin{sublem}
\label{lem:ctriplep}
Let $(C;\{E^{(N)}_p \}_{p\in {\mathbb Z}_N};C^*;\{E^{*(N)}_p \}_{p\in {\mathbb Z}_N})$
denote a cyclic TD system. Then the following holds for ${p,m,r\in {\mathbb Z}_N}$  and $r<N/2$:
\begin{enumerate}
\item[\rm (i)] $E^{*(N)}_pC^rE^{*(N)}_m=0$ \quad if\quad $r<|p-m|<N-r$\ ,
\item[\rm (ii)] $E^{(N)}_pC^{*r}E^{(N)}_m=0$ \quad if \quad $r<|p-m|<N-r$\ .
\end{enumerate}
\end{sublem}


We assume ${\mathbb K}={\mathbb C(q)}$, for $q$ an indeterminate. Given a TD pair $(A,A^*)$ of $q-$Racah type, recall that  we have the decompositions 
\begin{equation}\label{decomp-q}
V=\bigoplus_{n=0}^d V_n(q) \qquad \mbox{or} \qquad V=\bigoplus_{n=0}^d V^*_n(q),
\end{equation}
where $V_n(q)$ (resp. $V^*_n(q)$) denotes the $A$-eigenspace
(resp. $A^*$-eigenspace) associated with the eigenvalue $\theta_n(q)$  (resp. $\theta^*_n(q)$) for $0\leq n\leq d$ given by (\ref{specgen-I}). According to (\ref{eq:t1}), with respect to an $A$-eigenspace the action of $A^*$ is:
\begin{equation}
A^* V_n(q) \subseteq V_{n-1}(q) + V_n(q)+ V_{n+1}(q) \qquad \qquad 0 \leq n \leq d,
\label{eq:t1-q}
\end{equation}
where $V_{-1}(q) = 0$ and $V_{d+1}(q)= 0$. For the rest of our discussion we will use the following assumption.
\begin{ass}\label{ass}
Let $A,A^*$ be a TD pair of q-Racah type of diameter $d$.
We assume 
there exist two bases in $V$  where the matrix representing 
$A$ is diagonal (resp. block tridiagonal)  and the one representing  $A^*$ is block tridiagonal (resp. diagonal) and such that, specializing $q$ to $e^{\rmi\pi/N}$, we have\footnote{We thank  P. Terwilliger  for stressing us the importance of the point (ii).}
\begin{itemize}
\item[(i)] all entries of $A$ and $A^*$ are finite with respect to the bases,
\item[(ii)] all entries of the transition matrix between the  two bases are finite,
\item[(iii)] $b,c,b^*,c^*$ from~\eqref{specgen-I} are such that $c/b\neq q^{2m}$ and $c^*/b^*\neq q^{2m'}$ for any  $m,m'\in \overline{0,N-1}$. 
\end{itemize}
\end{ass}

\begin{rem}
Two different choices of the pairwise bases from Assumption~\ref{ass} may not give the same pair of linear transformations under the specialization  $q$ to $e^{\rmi\pi/N}$, we do not study this important question here. But our statements about properties of cyclic TD pairs obtained from TD pairs of $q-$Racah type do not depend on such a choice.
\end{rem}

From now on, we abuse  the convention that $q$ stands for an indeterminate  by writing $q=e^{\frac{\rmi\pi}{N}}$ for the specialization of the indeterminate $q$ to the root of unity $e^{\frac{\rmi\pi}{N}}$.
A large family of examples satisfying Assumption~\ref{ass} can be explicitly constructed. A simple example is
given in Appendix~\ref{app:q-Racah}. 

\begin{Prop}\label{lem1}  For $q=e^{\frac{\rmi\pi}{N}}$ and integer  $2\leq N\leq d$, the operators $A$ and $A^*$ from Assumption~\ref{ass} have $N$ distinct eigenvalues. They read:
\begin{eqnarray}\label{spec-root}
\begin{split}
\theta_t &= a + b q^{2t} + c q^{-2t}, 
\\
\theta^*_t &= a^* + b^* q^{2t} + c^* q^{-2t},
\end{split}
 \qquad \mbox{for}  \quad t=0,1,...,N-1.
\end{eqnarray}
Let $V_t^{(N)}$ (resp. $V_t^{*(N)}$) denotes the eigenspace associated with $\theta_t$ (resp. $\theta^*_t$). One has the decomposition:
\beqa
V_t^{(N)}= \bigoplus_{k=0}^{\lfloor\frac{d}{N}\rfloor} V_{t+kN}(q)\Bigl|_{q=e^{\frac{\rmi\pi}{N}}} \qquad \text{and}
 \qquad V_t^{*(N)}= \bigoplus_{k=0}^{\lfloor\frac{d}{N}\rfloor} V^{*}_{t+kN}(q)\Bigl|_{q=e^{\frac{\rmi\pi}{N}}}\label{esp}
\eeqa
with the action
\beq\label{tdroot1}
A^*\, V_t^{(N)} \subseteq V_{t+1}^{(N)} + V_t^{(N)} + V_{t-1}^{(N)}  \qquad \qquad (0 \leq t \leq N-1),
\eeq
where $V_{-1}^{(N)} = V_{N-1}^{(N)}$ and $V_{N}^{(N)} = V_{0}^{(N)}$
and
\beq\label{tdroot2}
A V_t^{*(N)} \subseteq V_{t+1}^{*(N)} + V_t^{*(N)} + V_{t-1}^{*(N)}  \qquad \qquad (0 \leq t \leq N-1),
\eeq
where $V_{-1}^{*(N)} = V_{N-1}^{*(N)}$ and $V_{N}^{*(N)} = V_{0}^{*(N)}$.

\vspace{1mm}
For $N>d$, the operator $A$ (resp. $A^*$) from Assumption~\ref{ass} has $d+1$ distinct eigenvalues $\theta_t$ (resp.~$\theta^*_t$)  and the corresponding eigenspaces $V^{(N)}_t=V_t(e^{\frac{\rmi\pi}{N}})$ (resp. $V^{*(N)}_t=V^*_t(e^{\frac{\rmi\pi}{N}})$), and  acts as in~(\ref{eq:t2}) (resp.~(\ref{eq:t1})).
\end{Prop}
\begin{proof} First, consider $2\leq N\leq d$ and recall the expressions  (\ref{specgen-I}) with $b,c,b^*,c^*\neq 0$. The conditions  $\theta_t-\theta_{t'}= 0$ and $\theta^*_t-\theta^*_{t'}=0$ for some $t\neq t'$ reduce to the equalities $c/b= q^{2m}$ and $c^*/b^*= q^{2m'}$,  for some $m,m'\in \overline{0,N-1}$, which are excluded  by Assumption~\ref{ass} (iii). So, the eigenvalues are  distinct.
Under  Assumption~\ref{ass}, the eigenvectors of $A$ are well-defined 
at $q=e^{\frac{\rmi\pi}{N}}$ for any positive integer $N>1$, and a vector in $V_n(e^{\rmi\pi/N})$ corresponds to the eigenvalue $\theta_n(e^{\rmi\pi/N})$. For $0\leq n \leq \dia$, there exist unique non-negative integers $t,k$ such that $n=t+kN$, with $0\leq t \leq N-1$. Using (\ref{specgen-I}), we immediately find (\ref{spec-root}).
Then for each fixed~$t$, we note that by Assumption~\ref{ass} the eigenspace $V^{(N)}_t$ of the operator $A$ is spanned by the eigenvectors in $V_t(e^{\frac{\rmi\pi}{N}}),V_{t+N}(e^{\frac{\rmi\pi}{N}}), \ldots$ and thus we have the decomposition in~\eqref{esp}. 
 We have obviously the similar statement for  the  $A^*$ eigenspaces  $V^{*(N)}_t$ of the eigenvalue $\theta^*_t$. The action~\eqref{tdroot1}-\eqref{tdroot2} of $A$ and $A^*$ on these spaces follows from
 our Assumption~\ref{ass} where the  specialization of  $A^*$  with finite entries means that the action is block tridiagonal on each $V_n(e^{\rmi\pi/N})$ and thus by construction on  $V^{(N)}_t$.

 For $N>d$, we have obviously $d+1$ distinct eigenvalues $\theta_t$ (resp. $\theta^*_t$) by Assumption (\ref{ass}) (iii) and each $V_t^{(N)}$ is $V_t(e^{\frac{\rmi\pi}{N}})$ (resp. $V_t^{*(N)}$ is $V^*_t(e^{\frac{\rmi\pi}{N}})$). The statement then follows.
\end{proof}
\vspace{1mm}

We emphasize that given a  TD pair of $q-$Racah type $A, A^*$ its specialization to $q$ a root of unity does not necessary give an irreducible action of $A$ and $A^*$.
 Studying the irreducibility is a difficult problem.
In Appendix~\ref{app:q-Racah}, we give an explicit example of a cyclic TD pair obtained from a TD pair of $q-$Racah type with $a=a^*=0$, where
we give a rather technical proof of the irreducibility.
Combining the results in Proposition~\ref{lem1} and Assumption~\ref{ass} (ii)\footnote{The point (ii) about existence of the specialization of the transition matrix between the two eigenbases allows us to show that the two pairs of matrices $(A,A^*)$ specialized in the first eigenbasis and in the second eigenbasis, correspondingly, give the same  pair of linear maps $A:V\to V$ and $A^*:V\to V$.}, we have the following corollary.
\begin{cor}\label{cor}
For $q=e^{\frac{\rmi\pi}{N}}$ and integer  $2\leq N\leq d$, assume the vector space $V$ is irreducible under the action of $A$ and $A^*$ in Assumption~\ref{ass}. Then, the  pair of operators $A$, $A^*$ form a cyclic tridiagonal pair of cyclicity $N$.
\end{cor}

Note that our first example of cyclic TD pairs given in Appendix~\ref{app:first-ex} does not appear as a root-of-unity specialization  of a TD pair of $q-$Racah type. 

\subsection{Cyclicity and the minimal polynomials}
From now on, we will assume that $2\leq N\leq d$. 
Then according to Proposition~\ref{lem1}, any cyclic TD pair $A$, $A^*$ obtained from a TD pair of $q-$Racah type of diameter $d$ has exactly $N$ distinct eigenvalues denoted by $\theta_t$ and~$\theta^*_t$, respectively, for $t=0,1,...,N-1$. As a consequence, the operator $A$  satisfies an algebraic equation determined by the minimal polynomial
\beqa\label{P-def}
P_N\bigl(x;\{\theta_t\}\bigr)=\prod_{t=0}^{N-1}(x-\theta_t),
\eeqa
and similarly for $A^*$, and $\{\theta_t\}$ stands for the set $\{\theta_t\}_{t=0}^{N-1}$ for brevity (and we follow this convention below).

\begin{subprop}\label{prop:min-poly}
Assume $q=e^{\frac{\rmi\pi}{N}}$. Let $A$, $A^*$ denote a cyclic TD pair with spectra of the form~(\ref{spec-root}). Then, the operators $A$ and $A^*$ satisfy the relations
\beqa
P_N\bigl(A;\{\theta_t\}\bigr)=0\qquad \text{and} \qquad P_N\bigl(A^*;\{\theta^*_t\}\bigr)=0. \label{relPt}
\eeqa
\end{subprop}
\vspace{1mm}

Note that the minimal polynomials can be expanded as follows.
Let $e_k(\theta_0,\cdots,\theta_{N-1})$, with $k=0,1,...,N-1$, denote the elementary symmetric polynomials in the variables $\theta_n$, i.e. given by
\beqa
e_k(\theta_0,\cdots,\theta_{N-1})=\!\!\!\!\!\sum_{0\leq j_1< j_2< \cdots < j_k\leq N-1}\!\!\!\!\!\!\!\!\!\!\!\!\! \theta_{j_1}\theta_{j_2}\cdots  \theta_{j_{k}} \quad \mbox{with} \quad e_0(\theta_0,\cdots,\theta_{N-1})=1.
\eeqa
For $N\geq 2$, we then have
\beqa
P_N\bigl(x;\{\theta_t\}\bigr) = \sum_{k=0}^N e_k(\theta_0,\cdots,\theta_{N-1})x^{N-k}\  \label{polymin}
\eeqa
and similarly for $P_N\bigl(x;\{\theta^*_t\}\bigr)$ with the replacement $(a,b,c)\rightarrow (a^*,b^*,c^*)$.
For instance\footnote{For $N=1$, $q^2=1$ there is no polynomial. In this special case, the operators $\cW_0,\cW_1$ however reduce to elements that generate the undeformed Onsager algebra.}: \vspace{1mm} 
\begin{itemize}
\item For $N = 2,~ q^4 = 1$,
\begin{equation*}
P_2(x;\theta_0,\theta_1) = x^2 - 2ax + (a+b+c)(a-b-c);
\end{equation*}
\item For $N =3,~ q^6 = 1$,
\begin{equation*}
P_3(x;\theta_0,\theta_1,\theta_2) =x^3-3ax^2+3(a^2-bc)x-(a+b+c)(a^2+b^2+c^2-ab-bc-ac); 
\end{equation*}
\item For $N = 4,~ q^8 =1$,
\begin{align*}
P_4(x;\theta_0,\theta_1,\theta_2,\theta_3) &= x^4-4ax^3 +2(3a^2-2bc)x^2-4a(a^2-2bc)x\qquad\\
&\qquad\qquad + (a+b+c)(a-b-c)(a^2+b^2+c^2-2bc).
\end{align*}
\end{itemize}
\vspace{1mm}

\subsection{The divided polynomials}\label{sec:div-poly}
For  two operators $A$ and $A^*$ that form a cyclic tridiagonal pair with spectra~\eqref{spec-root}, for $q$ specialized at a root of unity, we now introduce and study  two additional operators that are called below the divided polynomials. \vspace{1mm}

In view of the fact that the polynomials (\ref{polymin}) for $A$ and $A^*$ vanish for $q^{2N}=1$, see Proposition~\ref{prop:min-poly}, we renormalise them by the $q$-factorial $[N]_q!$ and take then the root of unity limit called the {\it divided polynomials} for the cyclic TD pair $A$, $A^*$: 
\beqa\label{WNnew}
\AN\, = \lim_{q\rightarrow e^{\frac{\rmi\pi}{N}}} \ \frac{P_N(A;\{\theta_t\})}{[N]_q!} \qquad 
\text{and} \qquad
\ANs\, = \lim_{q\rightarrow e^{\frac{\rmi\pi}{N}}} \ \frac{P_N(A^*;\{\theta^*_t\})}{[N]_q!}.  
\eeqa
Here,  the limit is taken in the sense of matrix elements and the entries of $A$ and $A^*$ are (rational) functions in $q$, we also assume that $\theta_t$ in~\eqref{P-def}  is fixed for given $N$ and does not depend on $q$, i.e. taken as the eigenvalue at the root of unity. We now give a general proof that~\eqref{WNnew} is well defined in the eigenbases of $A$ and $A^*$, respectively.\vspace{1mm}

According to the spectral properties~\eqref{spec-root} of the cyclic tridiagonal pair $A$, $A^*$, the eigenvalues of the divided polynomials (\ref{WNnew}) can be derived in a closed form. Let us first consider the special case $N=2$ which is the most degenerate situation. The spectra are particularly simple.
\begin{sublem}\label{lem3} Let $d$ be the diameter of a TD pair of $q-$Racah type.  Assume that the derivatives $\frac{\partial {\theta}_n(q)}{\partial q}$ (resp. $\frac{\partial {\theta}^*_n(q)}{\partial q}$) are finite at $q=\rmi$, for $n=0,1,...,d$.
For $N = 2$, the eigenvalues $\tilde{\theta}_n$, with $n = 0, \dots, d$ (resp. $\tilde{\theta}_n^*$, $n = 0, \dots, d$) of the divided polynomials $\Aex{2}$ (resp. $\Aexs{2}$) take the form
\begin{equation}
\tilde{\theta}_n = \rmi(c^2-b^2)2n \quad  \mbox{and} \quad \tilde{\theta}^*_n = \rmi({c^*}^2-{b^*}^2)2n.\label{eval2}
\end{equation}
\end{sublem}
\begin{proof}
For $N =2$, the minimal polynomial  $P_2(x;\theta_0,\theta_1)$ is given below (\ref{polymin}).  Then, evaluating the right-hand side of the first expression~\eqref{WNnew} in the $A$-eigenbasis in $V_n(\rmi)$ we can replace $A$ simply by its eigenvalue and define the numbers
\begin{eqnarray*}
\tilde{\theta}_n = \lim\limits_{q \to e^{\rmi\pi/2}}{\frac{{\theta_n}^2-2a\theta_n+(a+b+c)(a-b-c)}{q+q^{-1}}}.
\end{eqnarray*}
Using the l'H\^opital's rule, we obtain
\begin{eqnarray*}
\tilde{\theta}_n& =& \lim\limits_{q \to e^{\rmi\pi/2}}{\frac{4nq^{-1}(bq^{2n}+cq^{-2n})(bq^{2n}-cq^{-2n})}{1-q^{-2}}}
\end{eqnarray*}
which reduces to (\ref{eval2}). A similar expression is obtained for $\tilde{\theta}_n^*$, replacing $(b,c)\rightarrow (b^*,c^*)$.
\end{proof}
\vspace{1mm}

The analysis above can be generalized in a straightforward manner. 
\begin{sublem}\label{lem2} 
Let $d$ be the diameter of a TD pair of $q-$Racah type. Assume that the derivatives $\frac{\partial {\theta}_n(q)}{\partial q}$ (resp. $\frac{\partial {\theta}^*_n(q)}{\partial q}$) are finite at $q=e^{\frac{\rmi\pi}{N}}$, for $n=0,1,...,d$ and integer $N >2$.
Then,  the  divided polynomials  $\AN$ (resp. $\ANs$) have $d+1$ eigenvalues denoted by $\tilde{\theta}_n$ (resp. $\tilde{\theta}^*_n$), $n=0,1,...,d$. If $k$, $t$ are non-negative integers   such that $n=t+kN$, $0\leq t \leq N-1$, the eigenvalues read as
\begin{equation}
\tilde{\theta}_n = C_0^{(n)}\left(\frac{\partial a}{\partial q}q+\frac{\partial b}{\partial q}q^{2n+1}+\frac{\partial c}{\partial q}q^{-2n+1}+ 2n( b q^{2n} - c q^{-2n})\right)\Bigl|_{q=e^{\frac{\rmi\pi}{N}}}    \label{spectnew}
\end{equation}
and
\begin{equation}
C_0^{(n)}=-\frac{(q-q^{-1})}{2N[N-1]_q!}  \prod_{j=0,j\neq t}^{N-1} (\theta_t-\theta_j),\nonumber
\end{equation}
where $q=e^{\frac{\rmi\pi}{N}}$.
A similar expression is obtained for $\tilde{\theta}_n^*$, replacing $(a,b,c)\rightarrow (a^*,b^*,c^*)$.
The eigenspace of $\AN$ corresponding to $\tilde{\theta}_n$ is $V_n\equiv V_n(q = e^{\frac{\rmi\pi}{N}})$, and similarly for $\ANs$.
\end{sublem}
\begin{proof}
 In terms of the  indeterminate $q$,   recall that $A$  has $d+1$ distinct eigenvalues 
%
\begin{equation*}
\theta_n(q)=a+bq^{2n}+cq^{-2n}, \qquad n=0,\dots,d,
\end{equation*}
 and corresponding eigenspaces $V_0(q)$, $V_1(q)$, $\dots$, $V_d(q)$.
 By Proposition~\ref{lem1} we have that $A$ has $N$ distinct eigenvalues $\theta_0, \theta_1, \dots, \theta_{N-1}$ at the specialization $q=e^{\frac{\rmi\pi}{N}}$.
Let $\tilde{{\theta }}_n$ (resp. $\tilde{{\theta }}_n^*$) denote the eigenvalues of $\AN$ (resp. $\ANs$).
We calculate the spectrum of the right-hand side of the first expression~\eqref{WNnew} in the $A$-eigenbasis. Restricting to the subspace $V_n( e^{\frac{\rmi\pi}{N}})$, we can replace $A$  by its eigenvalue $\theta_n$. Therefore, 
$\tilde{\theta}_n$ can be then written as
\beqa
\tilde{{\theta}}_{n}=\mathop{\lim}\limits_{q\to e^{\frac{\rmi\pi}{N}}}\frac{(\theta_n(q)-\theta_0)(\theta_n(q)-\theta_1)\dots(\theta_n(q)-\theta_{N-1})}{[N]_q!},\qquad n =0,\dots, d. \nonumber
\eeqa
For all $n = 0,\dots, d$, there exist unique $k, t$ non-negative integers, such that $n = kN +t,~~ 0\le t\le N-1$. It yields to
\begin{equation}
\label{eq1}
\tilde{\theta}_n=\frac{(q-q^{-1})\prod\limits_{j=0,j\ne t}^{N-1}{(\theta_t-\theta_j)}}{[N-1]_q!}\mathop{\lim}\limits_{q\to e^{\frac{\rmi\pi}{N}}}\frac{\theta_n(q)-\theta_t}{q^N-q^{-N}}.
\end{equation}
Using l'H\^opital's rule, one has
\begin{eqnarray*}
&&\mathop{\lim}\limits_{q\to e^{\frac{\rmi\pi}{N}}}{\frac{\theta_n(q)-\theta_t}{q^N-q^{-N}}}=\mathop{\lim}\limits_{q\to e^{\frac{\rmi\pi}{N}}}\frac{\frac{\partial \theta_n(q)}{\partial q}     }{-2Nq^{-1}}.
\end{eqnarray*}
Plugging this expression into (\ref{eq1}), the equation (\ref{spectnew}) follows.  A similar expression is obtained for $\tilde{\theta}_n^*$, replacing $(a,b,c)\rightarrow (a^*,b^*,c^*)$.
Finally  by the construction, the eigenspaces of $\AN$ and $\ANs$ corresponding to $\tilde{\theta}_n$ and  $\tilde{\theta}^*_n$ are  $V_n( e^{\frac{\rmi\pi}{N}})$ and $V^*_n( e^{\frac{\rmi\pi}{N}})$, respectively.
\end{proof}

\begin{rem}\label{rem:Hopital}
 We note that the assumption on the first derivatives in the Lemmas~\ref{lem3} and~\ref{lem2} is only a technical assumption, to make the formulas less complicated. One can easily generalize the statement by assuming that the  higher derivatives $\frac{\partial^k {\theta}_n(q)}{\partial q^k}$ (resp. $\frac{\partial^k {\theta}^*_n(q)}{\partial q^k}$) for some finite~$k$ are finite at $q=e^{\frac{\rmi\pi}{N}}$, for $n=0,1,...,d$. The expression for $\tilde{\theta}_n$ and $\tilde{\theta}^*_n$ is then obtained by applying l'H\^opital's rule $k$ times. In Section~\ref{sec5}, we will have a situation where the first derivatives are diverging but the second ones are finite.
\end{rem}

Using the expressions (\ref{WNnew}) of the divided polynomials  in terms of $A$ and $A^*$, we now consider the action of the two operators on the eigenspaces $V_n\equiv V_n(e^{\frac{\rmi\pi}{N}})$ and $V^*_n\equiv V^*_n(e^{\frac{\rmi\pi}{N}})$ of $\AN$ and $\ANs$, respectively.
\begin{Prop}\label{actWN} 
For any eigenbasis of $A$ (resp. $A^*$) where $\ANs$ (resp. $\AN$) is finite,
 the divided polynomials $\ANs$ and $\AN$  act as
\beqa
&&\AN V_n \subseteq V_n\ ,\label{Nstruct}\\ 
&& \ANs V_n \subseteq V_{n-N} +  \cdots + V_n+  \cdots + V_{n+N} \qquad \qquad (0 \leq n \leq d),\nonumber\\
&& \ANs V^*_s \subseteq V^*_{s}\ ,\nonumber\\ 
&&\AN V^*_s \subseteq V^*_{s-N} + \cdots + V^*_s+  \cdots + V^*_{s+N} \qquad \qquad (0 \leq s \leq d).\nonumber
\eeqa
\end{Prop}
\begin{proof} The first and third equations follow from Lemmas~\ref{lem3} and~\ref{lem2}. 
In any eigenbasis of $A$ for indeterminate $q$ the polynomial $P_N(A^*;\{\theta^*_t\})/[N]_q!$ in $A^*$ is  of order $N$ (recall~\eqref{polymin})  and acts as a block $(2N+1)$-diagonal matrix. We consider then any eigenbasis of $A$ from Assumption~\ref{ass} and where $\ANs$  is finite and therefore it acts as a block $(2N+1)$-diagonal matrix as well. Therefore, it acts as in the second line of~\eqref{Nstruct}.
Similar argument gives the proof for the action of the divided polynomial for $A$.
\end{proof}
Note that in the basis which diagonalizes $A$  (resp. $A^*$), the matrix representing  $\ANs$ (resp. $\AN$) is  block $(2N+1)$-diagonal\footnote{An example of such operators is given in Appendix~\ref{app:ex-2-div}.}.
From this point of view, the two operators (\ref{WNnew}) can be understood as objects generalizing the properties of tridiagonal pairs.
According to the relationship between TD pairs and the representation theory of TD algebras (see Section 2), an important problem is to identify the relations satisfied by the four operators. In the next Section, we investigate such relations for a rather large class of cyclic TD pairs and corresponding divided polynomials.\vspace{1mm}

\section{The algebra generated for a class of cyclic pairs}

The purpose of this section is to identify the set of relations satisfied by a class of cyclic tridiagonal pairs and the corresponding divided polynomials
\beqa
{\NtcWi} = \mbox{lim}_{q\rightarrow e^{\rmi \pi/N}} \ \frac{P_N\bigl(\cW_i;\{\theta_t^{(i)}\}\bigr)}{[N]_q!}\ , \qquad \mbox{with} \qquad  i\in\mathbb{Z}_2\ ,\label{WNnew2}
\eeqa
where the operators $\cW_0$ and $\cW_1$ form a TD pair of $q-$Racah type, as discussed in Section 2. The class of cyclic TD pair  considered  here is associated with operators having the spectra of the form (\ref{spec-root}), with $a=a^*=0$. First, we identify the relations for $N=2$ and then for $N>2$. All together, the four operators generate a new algebra. We  conjecture that the relations we present here  (Theorems~\ref{prop:mixed-rel-k-zero} and~\ref{propNOns})   are the defining relations for a specialization of the $q-$Onsager algebra at the root of unity. We also show that the subalgebra generated by $\NtcWzero$, $\NtcWun$ is a higher-order generalization of the (classical) Onsager algebra with defining relations  given by~(\ref{genOAN}). 

\subsection{The relations for $N=2$}
We study here the algebra generated by the cyclic TD pair $\W_0,\W_1$
and the divided polynomials (\ref{WNnew2}) at $N=2$ for $a=a^*=0$. The generators  $\W_i$ satisfy  the $q-$Dolan-Grady relations  (\ref{Talg}) evaluated  at $q=e^{\rmi\pi/2}$ or, because $\rho=\rho^*=0$ (see (\ref{rhored})), the $q$-Serre relations at $q^2=-1$:
\beq\label{q-Serre-N2}
\W_i^3\W_{i+1} + \W_i^2\W_{i+1} \W_i -\W_i\W_{i+1}\W_i^2 -\W_{i+1} \W_i^3 =0
\eeq
together with
the polynomial relations~\eqref{relPt} that take the form
\beqa\label{nilp-rel}
\W_0^2 = (b+c)^2 \one \quad \mbox{and}\quad \W_1^2 = (b^*+c^*)^2 \one
\eeqa
Note that the relations~\eqref{q-Serre-N2} are actually consequences of~\eqref{nilp-rel}, i.e., they are identities if we take into account~\eqref{nilp-rel} only. It simply means that, for the special case $N=2$ the basis elements in the subalgebra generated by $\W_i$ are monomials of the form
\beqa
\W_i\W_{i+1}\W_i\W_{i+1}\cdots\nonumber
\eeqa

For the mixed relations (between $\W_i$ and $\NWex{2}_{i+1}$) we obtain:
\beqa
\W_i^3\NWex{2}_{i+1} + \W_i^2\NWex{2}_{i+1} \W_i -\W_i\NWex{2}_{i+1}\W_i^2 -\NWex{2}_{i+1} \W_i^3 &=&0,
\label{eq:WN2-mix-rel-1}\\
\qquad \Bigl(\NWex{2}_{i}\Bigr)^3\W_{i+1} -3\Bigl(\NWex{2}_{i}\Bigr)^2\W_{i+1} \NWex{2}_{i} 
+ 3 \NWex{2}_{i}\W_{i+1} \Bigl(\NWex{2}_{i}\Bigr)^2 - \W_{i+1}\Bigl(\NWex{2}_{i}\Bigr)^3 &=& \rho_{i}^{[2]}\Bigl[\NWex{2}_{i},\W_{i+1}\Bigr],\label{eq:WN2-mix-rel-2}
\eeqa
where
\begin{equation}\label{rho-i}
\rho_0^{[2]} = - 4(b^2-c^2)^2 , \qquad \rho_1^{[2]} = - 4({b^*}^2-{c^*}^2)^2.
\end{equation}
Finally, we have also relations between the divided polynomials
 $\NWex{2}_0$ and $\NWex{2}_1$: 
\begin{multline}\label{eq:WN2-rel}
\Biggl[\,\NWex{2}_i\,,
\Bigl(\NWex{2}_{i}\Bigr)^4 \NWex{2}_{i+1} +\NWex{2}_{i+1}\Bigl(\NWex{2}_{i}\Bigr)^4  \\
-4\Bigl(\NWex{2}_{i}\Bigr)^3 \NWex{2}_{i+1}\NWex{2}_{i}
+6\Bigl(\NWex{2}_{i}\Bigr)^2 \NWex{2}_{i+1}\Bigl(\NWex{2}_{i}\Bigr)^2
-4 \NWex{2}_{i}\NWex{2}_{i+1}\Bigl(\NWex{2}_{i}\Bigr)^3 \\
-5\rho_i^{[2]}\Bigl\{\Bigl(\NWex{2}_{i}\Bigr)^2 \NWex{2}_{i+1} -2\NWex{2}_{i} \NWex{2}_{i+1}\NWex{2}_{i} + \NWex{2}_{i+1} \Bigl(\NWex{2}_{i}\Bigr)^2 \Bigr\}
+4\bigl(\rho_i^{[2]}\bigr)^2  \NWex{2}_{i+1}
\,\Biggr]=0.
\end{multline}

The proof of the relations (\ref{eq:WN2-mix-rel-1}), (\ref{eq:WN2-mix-rel-2}) and (\ref{eq:WN2-rel}) will be given below for $N>2$.\vspace{1mm}

\subsection{The relations for $N>2$}
For the indeterminate $q$,  a TD pair of $q-$Racah type
with spectra (\ref{specgen-I}) for $a=a^*=0$  satisfies the $q-$Dolan-Grady relations~(\ref{Talg}) \cite{Ter03}, as well the higher-order $q-$Dolan-Grady relations (\ref{qDGfinr}) derived in \cite{BV} (see Section 2). In particular, for  the specialization $q=e^{\frac{\rmi\pi}{N}}$ it implies that any cyclic TD pair $C,C^*$ that is obtained from such a TD pair  satisfies (\ref{Talg}), (\ref{qDGfinr}) with $\W_0\rightarrow C$, $\W_1 \rightarrow C^*$. Below, we present an independent proof of these relations. 
\begin{prop}
Assume $q=e^{\frac{\rmi\pi}{N}}$. Let $\cW_0,\cW_1$ denote a cyclic TD pair with spectra given by (\ref{spec-root}) with $a=a^*=0$. Then, the operators $\cW_0,\cW_1$ satisfy the relations (\ref{Talg}).
\end{prop}
\begin{proof}
Let $\Delta_1$ denote the expression of the left-hand side minus the  right-hand side  of the first equation in~(\ref{Talg}).  We show now that $\Delta_1=0$. Recall that $q=e^{\frac{\rmi\pi}{N}}$ and  $\cW_0,\cW_1$ is a cyclic TD pair by the assumption of the proposition, we thus use the primitive idempotents defined in~\eqref{EN-def}. For $i,j\in {\mathbb Z}_N$, one has $E^{(N)}_i \Delta_1 E^{(N)}_j = p_1(\theta_i,\theta_j) \ E^{(N)}_i \cW_1 E^{(N)}_j$ with
\beqa
p_1(x,y) = (x-y) (x^2 - (q^2+q^{-2}) xy +y^2 -\rho)\  .
\eeqa
For each pair $i,j$ it is straightforward to check that $p_1(\theta_i,\theta_j)=0$ if  $|i-j|\leq 1$ or $|i-j|=N-1$, while according to Lemma \ref{lem:ctriplep} we have $E^{(N)}_i \cW_1 E^{(N)}_j =0$  if  $1<|i-j|< N-1$. It implies  $\Delta_1=0$ and thus the first equation in (\ref{Talg}) holds. Similar arguments are used to show the second equation in (\ref{Talg}).
\end{proof}

Similar arguments are used to derive the higher-order $q-$Dolan-Grady relations for the cyclic TD pair.
\begin{prop}\label{prop-qOA}
Assume $q=e^{\frac{\rmi\pi}{N}}$. Let $\cW_0,\cW_1$ denote a cyclic TD pair with spectra given by (\ref{spec-root}) with $a=a^*=0$. Let $r$ be a positive integer such that $1\leq r \leq N-1$. Then, the operators $\cW_0,\cW_1$ satisfy the $r$-th  higher-order $q-$Dolan-Grady relations~(\ref{qDGfinr}).
\end{prop}
\begin{proof} 
Let $\Delta_r$ denote the expression of the left-hand side of the first equation in (\ref{qDGfinr}). We show $\Delta_r=0$. For $i,j\in \mathbb{Z}_N$, one has $E^{(N)}_i \Delta_r E^{(N)}_j = p_r(\theta_i,\theta_j) \ E^{(N)}_i \cW_1^r E^{(N)}_j $ with
\beqa
p_r(x,y) = (x-y)\prod_{s=1}^{r} \left(x^2- \frac{[2s]_{q^2}}{[s]_{q^2}}xy +y^2 - [s]^2_{q^2} \rho\right)\ . \label{polyqons}
\eeqa
For each $i,j$ it is straightforward to check that $p_r(\theta_i,\theta_j)=0$ if $ N-r\leq |i-j|\leq r$, while according to Lemma \ref{lem:ctriplep}  we have $E^{(N)}_i \cW_1^r E^{(N)}_j =0$  if $N-r>|i-j|> r$. It implies the first equation in (\ref{qDGfinr}). Similar arguments are used to show the second equation in (\ref{qDGfinr}).
\end{proof}

Note that any higher-order relation for $r\geq N$ can be reduced through the relations (\ref{relPt}) with $A\rightarrow \cW_0$, $A^*\rightarrow \cW_1$. Finally, recall that the two operators $\cW_0$ and $\cW_1$ act on a finite dimensional vector space. It implies that $\cW_0,\cW_1$ should satisfy additionnal polynomial relations with coefficients depending on the representation considered. These relations can be undertstood as generalizations of the Askey--Wilson relations. Examples of such relations can be found in \cite{BB1} for the  indeterminate~$q$.\vspace{1mm}

\subsubsection{Mixed relations}
Let us now turn to the derivation of the mixed relations, combining the cyclic TD pair and the divided polynomials. According to the action of the divided polynomials on the eigenspaces $V_n,V^*_s$ in (\ref{Nstruct}), polynomial relations are indeed expected. Below, for simplicity we focus on the class of cyclic TD pairs with spectra (\ref{spec-root}) and  $a=a^*=0$ and the rational  functions $b$, $c$, $b^*$, and $c^*$ of the form
\beqa\label{eq:conditions-bc}
b= b_0 + b_1 q^{\beta},\qquad 
c= c_0 + c_1 q^{-\beta},\qquad 
b^*= b^*_0 + b^*_1 q^{\beta^*},\qquad 
c^*= c^*_0 + c^*_1 q^{-\beta^*}\ ,
\eeqa
for $b_i, b_i^*, c_i, c_i^*\in\mathbb C$, for $i=0,1$, and integer $\beta$ and $\beta^*$. For this choice, the spectra (\ref{spectnew}) of the divided polynomials (\ref{WNnew2}) simplifies. For instance, using (\ref{spectnew}) the spectrum of $\NtcWzero$ simplifies to
\beqa
\tilde{\theta}_n=(2n +\beta) C_N, \quad n=0,1,...,\dia  
\eeqa
with
\beqa
 C_N = \frac{(q-q^{-1})^N}{2N} q^{-N(N+1)/2}\left(b^N - c^N\right) 
\label{specred}
\eeqa
and similarly for $\NtcWun$, replacing $C_N\rightarrow C_N^*$ and $(b,c)\rightarrow (b^*,c^*)$ in the above expressions. 
Importantly, the spectra have an arithmetic structure with respect to the integer $n=0,1,2,...,d$. 

\begin{subrem} We note that multiplying the generators $\cW_0$  and $\cW_1$ from
our basic example (\ref{q-O-XXZ-action-1}) in Appendix~\ref{app:q-Racah} by an overall factor $q-q^{-1}$, the spectra of the corresponding two  operators satisfy the assumption~\eqref{eq:conditions-bc} on the coefficient functions $a$, $b$, $c$, etc.
\end{subrem}

\vspace{1mm}

Now, to derive the set of relations satisfied by the cyclic TD pair and the divided polynomials,
we proceed by analogy with \cite{Ter03,BV}.
 \begin{subdefn}
 Let $x$ and $y$ denote commuting indeterminates. For each positive integer $N$ and  $q=e^{\frac{\rmi\pi}{N}}$, we define the two-variable polynomials $p^{(i)}_{k}(x,y)$, $i=0,1$, as follows:
 \beqa\label{defpoly}
 \qquad p^{(i)}_{k}(x,y) = (x-y)\prod_{s=1}^{k} (x^2- 2xy +y^2 -\rho^{[N]}_i s^2)\quad
 \eeqa
 with
 \beqa
 \rho^{[N]}_0=4\left(C_N\right)^2, \qquad \rho^{[N]}_1=4\left(C_N^*\right)^2. \label{defpoly-r}
 \eeqa
 Observe $p^{(i)}_{k}(x,y)$ have a total degree $2k+1$ in $x,y$.
 \end{subdefn}

 \begin{sublem}\label{polynul} Let $k$ be a positive integer. We  have
 \beqa
 p^{(0)}_{k}(\tilde{\theta}_r,\tilde{\theta}_s)=0, \qquad p^{(1)}_{k}(\tilde{\theta}_r^*,\tilde{\theta}_s^*)=0  \qquad \mbox{for} \qquad |r-s|\leq k\ .\nonumber
 \eeqa
 \end{sublem}
\begin{proof} For $r=s$, the relation obvioulsy holds. Due to the arithmetic progression of $\tilde{\theta}_s$ and $\tilde{\theta}_s^*$, one observes  $(2s)^2 -2(2s)(2s\pm 2k) + (2s\pm 2k)^2 =4k^2$. The two equalities then follow.
\end{proof}

\begin{thm}\label{prop:mixed-rel-k-zero}
Let $r$ be a positive integer. Then, the operators $\mathcal{W}_i$ and the divided polynomials $\NtcWi$, for $i\in\mathbb{Z}_2$, satisfy the mixed relations
\beqa
\sum_{p=0}^{N-1}{\sum\limits_{k=0}^{2N-1-2p}{(-1)^{p+k}\rho_i^p c_k^{[N-1,p]}\mathcal{W}_i^{2N-1-2p-k}\NtcWipun\mathcal{W}_i^k}} &=&0 \label{eq:mix1},\\
&&\label{eq:mix2} \\
\left(\NtcWi\right)^3\mathcal{W}_{i+1}-3\left(\NtcWi\right)^2\mathcal{W}_{i+1}\NtcWi+3\NtcWi\mathcal{W}_{i+1}\left(\NtcWi\right) -\mathcal{W}_{i+1}\left(\NtcWi\right)^3
\!\!\!\! &=&\!\!\!\!\rho_i^{[N]}\left[\NtcWi, \mathcal W_{i+1}\right],\nonumber
\eeqa
where \ $\rho_0=\rho$ and $\rho_1=\rho^*$ are given by (\ref{rhored})  and $c_k^{[N-1,p]}$ by~\eqref{cfinr} with $q=e^{\rmi \pi/N}$, and $\rho_i^{[N]}$ by~\eqref{defpoly-r}.
\end{thm}
\begin{proof} We give first a proof of (\ref{eq:mix1}) for $i=0$.
Let $\Delta^{(1)}$ denote the corresponding left-hand side of (\ref{eq:mix1}).  To show that $\Delta^{(1)}=0$, we recall that $q=e^{\frac{\rmi\pi}{N}}$ and  $\cW_0,\cW_1$ is a cyclic TD pair, we thus use the primitive idempotents $E_n^{(N)}$ defined in~\eqref{EN-def} which are projectors onto the eigenspaces $V_n^{(N)}$ in~\eqref{esp}.  Observe then 
\beqa
E_n^{(N)}\Delta^{(1)} E_m^{(N)} = p_{N-1}(\theta_n,\theta_m) E_n^{(N)} \NtcWun E_m^{(N)},\qquad 0\leq n,m\leq N-1,\nonumber
\eeqa
where the two-variable polynomial (\ref{polyqons}) has  to be expanded as a power series in the variable $x,y$. Recall then (\ref{polyqons}) with
 $r=N-1$,  we thus have  $p_{N-1}(\theta_n,\theta_m)=0$ for $|n-m|\leq N-1$. This  implies  $\Delta^{(1)}=0$.

We then prove (\ref{eq:mix2}) for $i=0$ in a similar way.
Let $\Delta^{(2)}$ denote the corresponding left-hand side of (\ref{eq:mix2}). To show $\Delta^{(2)}=0$, we use now the primitive idempotents $E_n$ which are the projectors onto the eigenspaces $V_n$ associated with the eigenvalues $\tilde{\theta}_n$. 
We have 
\beqa
E_n\Delta^{(2)} E_m = p_{1}^{(0)}(\tilde{\theta}_n,\tilde{\theta}_m) E_n \cW_1 E_m\nonumber
\eeqa
where the two-variable polynomial (\ref{defpoly}) has to be expanded as a power series in the variables $x$, $y$. According to~\eqref{eq:t1-q}
one has $E_n \cW_1 E_m=0$ for $|n-m|>1$. Using Lemma \ref{polynul} for $k=1$, it implies  $\Delta^{(2)}=0$. Similar arguments are used to show (\ref{eq:mix1}), (\ref{eq:mix2}) for $i=1$.
\end{proof}
\vspace{1mm}

Finally, we identify   relations satisfied by the divided polynomials.
\begin{thm}\label{propNOns} The divided polynomials
$\NtcWi$,  for  $i\in\mathbb{Z}_2$, satisfy the following relations:
\begin{eqnarray}
\sum_{p=0}^{N}\  \sum_{k=0}^{2N+1 -2p}  (-1)^{k+p}  \left(\rho_i^{[N]}\right)^{p}\, {d}_{k}^{[N,p]}\, \left(\NtcWi\right)^{2N+1-2p-k} \NtcWipun \left(\NtcWi\right)^{k}&=&0\ ,
\label{genOAhom}
\end{eqnarray}
with 
\begin{eqnarray}
d^{[N,p]}_k=
\binom{2N+1-2p}{k}
\sum\limits_{1\le s_1<\dots <s_p\le N}{s_1^2s_2^2\dots s_p^2} \qquad \mbox{with}\qquad k=0,1,...,N-p.
 \label{coeffq1}
\end{eqnarray}
\end{thm}
\begin{proof} Let $\Delta^{[N]}$  denote the expression on the left-hand side of (\ref{genOAhom}) for $i=0$. We
show $\Delta^{[N]}= 0$. Recall that $E_n$ denotes the projector onto the eigenspace $V_n$ associated with the eigenvalue $\tilde{\theta}_n$. We have then
\beqa
E_n\Delta^{[N]} E_m = p^{(0)}_{N}(\tilde{\theta}_n, \tilde{\theta}_m) E_n \NtcWun E_m .\nonumber
\eeqa
According to Proposition \ref{actWN}, one has $E_n \NtcWun E_m=0$ for $|n-m|>N$. Using Lemma \ref{polynul} for $k=N$, it implies  $\Delta^{[N]}=0$. Similar arguments are used to show (\ref{genOAhom})  for $i=1$.
\end{proof}

\begin{subrem} Note that the defining relations (\ref{genOAhom}) do not coincide with the higher-order $q-$Dolan-Grady relations (\ref{cfinr}) for $r=N$ and $q=1$. The main difference is the power of the operator $\NtcWipun$ entering in (\ref{genOAhom}) compared with (\ref{qDGfinr}) specialized at $q=1$.
\end{subrem}
\vspace{1mm}

Based on the previous results, we now introduce a new algebra that can be viewed as a higher-order generalization of the Onsager algebra. 
\begin{defn}\label{defOAN}
Let $N$ be a positive integer and $\rho_i^{[N]}, i\in {\mathbb Z}_2$, denote scalars in ${\mathbb K}$. The $N-$Onsager algebra $OA^{[N]}$ 
 is the associative ${\mathbb K}$-algebra with unit  generated by two elements $\NtW_0$, $\NtW_1$ subject to the relations 
\begin{eqnarray}
\sum_{p=0}^{N}\  \sum_{k=0}^{2N+1 -2p}  (-1)^{k+p}  \Bigl(\rho_i^{[N]}\Bigr)^{p}\, {d}_{k}^{[N,p]}\, 
\Biggl(\NtW_i\Biggr)^{2N+1-2p-k} \NtW_{i+1} \left(\NtW_i\right)^{k}&=&0,\label{genOAN}
\end{eqnarray}
with ${d}_{k}^{[N,p]}$ defined by (\ref{coeffq1}).
\end{defn}

Note that the classical Onsager algebra \cite{Ons,DG} corresponds to $N=1$ and $\rho_0^{[N]}=\rho_1^{[N]}=16$. \vspace{1mm}

\begin{cor}
The divided polynomials $\NtcWzero$ and $\NtcWun$ generate a (finite-dimensional) algebra which is a quotient
of  the $N-$Onsager algebra~$OA^{[N]}$ with $\rho_0^{[N]}$ and $\rho_1^{[N]}$ given in (\ref{defpoly-r}).
\end{cor}

\medskip
We finally conclude that the four operators $\W_i$ and $\NtcWi$, $i\in\mathbb{Z}_2$,  satisfy
\begin{enumerate}
\item the $q-$Dolan--Grady relations~\eqref{Talg} at $q=e^{\rmi\pi/N}$,
\item the polynomial relations
\beqa
P_N\bigl(\W_0;\{\theta_t\}\bigr)=0\qquad \text{and} \qquad P_N\bigl(\W_1;\{\theta^*_t\}\bigr)=0\ , \label{relPt-W}
\eeqa
where the polynomials $P_N$ are defined in~\eqref{P-def} together with~\eqref{spec-root},
\item  the mixed relations (\ref{eq:mix1}) and  (\ref{eq:mix2}), 
\item and  the $N-$Onsager algebra $OA^{[N]}$  relations (\ref{genOAhom}).
\end{enumerate}
\vspace{1mm}

 We also note that during the derivation of these relations we did not use the condition on  irreducibility of the action of the pair $\W_0$, $\W_1$. 
It is thus natural to define an (infinite-dimensional) algebra generated by the four generators subject to the  (1)-(4) defining relations.  Note that the algebra depends on 5 parameters: positive integer $N\geq 2$, and 4 complex parameters $b$, $c$, and $b^*$, $c^*$.  The first two relations give an analogue of the ``small quantum group'' for $U_q \widehat{s\ell}(2)$ at $q=e^{\frac{\rmi\pi}{N}}$ while the relations for the divided polynomials are analogues for the relations of the Lusztig's divided powers~\cite{Luszt} that form the classical algebra $U \widehat{s\ell}(2)$.
We thus conjecture that the algebra described in (1)-(4) is a specialization of the $q-$Onsager algebra at $q=e^{\frac{\rmi\pi}{N}}$.

\medskip

We finally note that our analysis in the last two sections can be easily extended to $q=e^{\rmi\pi m'/m}$ for rational $m'/m$ and coprime $m$, $m'$\,: the only conditions we actually use in Proposition~\ref{lem1} is $q^{2m}= 1$, so one has just to replace $N$ by $m$. The rest follows without changes.

\section{Orthogonal polynomials beyond Leonard duality revisited}\label{sec5}
All known one-variable orthogonal polynomials of the Askey scheme and their specializations satisfy a {\it bispectral problem} associated with  a three-term recurrence relation and a second-order differential or ($q-$)difference equation \cite{KS}. 
Generalizing Leonard's theorem \cite{Leo}, Bannai and Ito \cite{BI} showed furthermore that all orthogonal polynomials satisfying this type of  bispectral problem should  coincide with polynomials of the Askey scheme
including the so-called Bannai--Ito polynomials \cite{BI}. From the point of view of representation theory, all these orthogonal polynomials in one variable  find a natural interpretation within the framework of Leonard pairs \cite{Ter2}. Namely, for a pair of elements $A,A^*$ that act on an irreducible finite dimensional module and are diagonalizable, and satisfy the so-called Askey--Wilson relations 
\beqa
A^2A^* - (q^2+q^{-2}) AA^*A + A^*A^2 - \gamma (AA^*+A^*A) - \rho A^* &=& \gamma^* A^2+\omega A+\eta \one,\label{AWrel}\\
{A^*}^2A - (q^2+q^{-2})  A^*AA^* + A{A^*}^2 - \gamma^* (A^*A+AA^*) - \rho^* A &=& \gamma {A^*}^2+\omega A^*+\eta^* \one,\nonumber
\eeqa
with scalars $\beta,\gamma,\gamma^*,\rho,\rho^*,\omega,\eta,\eta^*$, it is known that there exist two bases with 
respect to which the two matrices representing $A,A^*$ are diagonal (resp. tridiagonal) and tridiagonal (resp. diagonal). 
Importantly, the overlap coefficients between the two bases coincide with the Askey--Wilson polynomials evaluated on a
 discrete support \cite{Zhed,Ter2}. They satisfy a three-term recurrence relation and a three-term difference equation \cite{Leo,BI}.

\subsection{Polynomials and  five-term difference equations}
In the literature, there are however isolated examples of one-variable orthogonal polynomials defined on a discrete support that {\it do not} satisfy  the Leonard duality property. Instead, the corresponding orthogonal polynomials defined on a discrete support satisfy a bispectral problem associated with a three-term recurrence  relation and a {\it five-term} difference equation. To our knowledge, the first example of such type is provided by the one-variable complementary Bannai--Ito polynomials $I_n(x)\equiv I_n(x;\rho_1,\rho_2,r_1,r_2)$ with scalars $\rho_1,\rho_2,r_1,r_2$ and variable $x$. They have been introduced in \cite{GVZ} through a Christoffel transformation of the Bannai--Ito polynomials \cite{BI}. By construction, the complementary Bannai--Ito polynomials solve the following bispectral problem. On one hand, they satisfy the three-term recurrence relation
\beqa
I_{n+1}(x) + (-1)^n\rho_2 I_n(x) + \tau_n I_{n-1}(x) = x I_n(x),\nonumber
\eeqa
where 
\beqa
\tau_{2p}&=& -\frac{p(p+\rho_1-r_1+1/2)(p+\rho_1-r_2+1/2)(p-r_1-r_2)}{(2p+g)(2p+g+1)},\nonumber\\
\tau_{2p+1}&=& -\frac{(p+g+1)(p+\rho_1+\rho_2+1)(p+\rho_2-r_1+1/2)(p-\rho_2-r_2+1/2)}{(2p+g+1)(2p+g+2)}\nonumber
\eeqa
with $g=\rho_1+\rho_2-r_1-r_2$. One the other hand, they are eigenfunctions of a second-order Dunkl shift operator. Namely, define $T^{\pm1}f(x)=f(x\pm 1)$ and the reflection operator $Rf(x)=f(-x)$. The complementary Bannai--Ito polynomials satisfy  a {\it five-term} difference equation of the form
\beqa
\cal{D}' I_n(x) = \tilde{\theta}'_n I_n(x)\label{opDprime}
\eeqa
where
\beqa
\cal{D}'=d'_1(x)T + d'_2(x)T^{-1} + d'_3(x)R + d'_4(x)TR -(d'_1(x)+d'_2(x)+d'_3(x)+d'_4(x))\one\nonumber
\eeqa
and the spectrum reads
\beqa
\qquad  \tilde{\theta}'_{2p}=p^2+(g+1)p,\quad   \tilde{\theta}'_{2p+1}=p^2+(g+2) +  \tilde{c}'_0.\label{spec-thet}
\eeqa
Here, $ \tilde{c}'_0$ is an arbitrary scalar. The corresponding explicit expressions for the rational functions $d'_i(x),i=1,...,4$, can be obtained from \cite[eq. (4.6)]{GVZ}.\vspace{2mm}

The second explicit example is given by the dual $-1$ Hahn polynomials $Q_{n}(x)\equiv Q_{n}(x;\rho_2,r_1,r_2)$ with scalars $\rho_2,r_1,r_2$ and variable $x$, considered either as a limiting case  $q\to -1$ of the dual $q-$Hahn polynomials  or as a limiting case $\rho_1\rightarrow \infty$ of the complementary Bannai--Ito polynomials \cite{TVZ,GVZ}. By construction,  they solve a bispectral problem which is defined as follows. They satisfy a  three-term recurrence relation:
\beqa
Q_{n+1}(x) +(-1)^n\rho_2 Q_{n}(x) + \sigma_n Q_{n-1}(x) =xQ_{n}(x)\nonumber
\eeqa
where
\beqa
\sigma_{2p}&=&-p(p-r_1-r_2),\nonumber\\
\sigma_{2p+1}&=&-(p+\rho_2-r_1+1/2)(p+\rho_2-r_2+1/2).\nonumber
\eeqa
The dual $-1$ Hahn polynomials also satisfy a {\it five-term} difference equation. They are the eigenfunctions of a difference operator of the Dunkl type:
\beqa
\cal{D}Q_{n}(x) =  \tilde{\theta}_n  Q_{n}(x)\label{opD}
\eeqa
where
\beqa
\cal{D}= d_1(x)T + d_2(x) T^{-1} + d_3(x)R + d_4(x)TR - (d_1(x)+d_2(x)+d_3(x)+d_4(x))\one \nonumber
\eeqa
and the spectrum reads \cite{GVZ}:
\beqa
  \tilde{\theta}_{2p} =p,\qquad    \tilde{\theta}_{2p+1} =p+\tilde{c}_0.\label{spec-nu}
\eeqa
Here, $\tilde{c}_0$ is an arbitrary scalar. The explicit expressions of the rational functions $d_i(x),i=1,...,4$, can be obtained from \cite[page 16]{GVZ}  (see also \cite{TVZ}).
Note that other known examples of one-variable polynomials beyond Leonard duality can be understood as descendants of the two above. For instance, the para-Krawtchouk and the symmetric Hahn polynomials can be seen as special cases of the complementary Bannai--Ito polynomials \cite{GVZ}.\vspace{1mm} 

\subsection{The complementary Bannai--Ito algebra}
Whereas the algebraic framework behind the orthogonal polynomials of the Askey scheme -- including the Bannai--Ito polynomials -- is encoded by the Askey--Wilson algebra with the defining relations (\ref{AWrel}) and the concept of Leonard pairs, the above examples  of orthogonal polynomials outside of Leonard duality arise within the representation theory of the so-called complementary Bannai--Ito algebra with two generators $\kappa_1,\kappa_2$ and involution~$r$. A presentation\footnote{In \cite{GVZ}, an alternative presentation is given in which case a third generator $\kappa_3=\big[\kappa_1,\kappa_2\big]$ is introduced.} of this algebra is given by the defining relations:
\beqa
\big[\kappa_1,r\big]=0,&& \{\kappa_2,r\}=2\delta_3, \qquad r^2=\one,\label{CBIalg}\\
\big[\kappa_1,\big[\kappa_1,\kappa_2\big]\big]&=& \frac{1}{2}\{\kappa_1,\kappa_2\} -\delta_2\big[\kappa_1,\kappa_2\big]r -\delta_3 \kappa_1 r + \delta_1\kappa_2 -\delta_1 \delta_3 r,\label{CBIalg-2}\\
\big[\kappa_2,\big[\kappa_2,\kappa_1\big]\big]&=& \frac{1}{2} \kappa_2^2 + \delta_2\kappa_2^2r + 2\delta_3\kappa_1r + 2\delta_3\big[\kappa_1,\kappa_2\big]r +\kappa_1 + \delta_4 r +\delta_5,\label{CBIalg-3}
\eeqa
where $\{x,y\}=xy+yx$ and $\delta_i$, for $i=1,\ldots,5$, are arbitrary scalars. Irreducible polynomial representations for the  complementary Bannai--Ito algebra (\ref{CBIalg}) have been constructed in \cite{GVZ}. For the complementary Bannai--Ito polynomials, the homomorphism
\beqa
\kappa_1\mapsto \cal{D}', \qquad \kappa_2 \mapsto x
\eeqa
with (\ref{opDprime}) is considered. In this case, the explicit expressions of the structure constants $\delta_i$, for $i=1,\ldots,5$, in terms of $r_1,r_2,\rho_1,\rho_2$ are given in \cite[eq. (5.6)]{GVZ}. In particular:
\beqa
\delta_1=\tilde{c}'_0(g+1-\tilde{c}'_0), \qquad \delta_2=g+3/2 -2\tilde{c}'_0\qquad \mbox{and} \qquad  \delta_3=\rho_2.\label{delta}
\eeqa
In this picture, the complementary Bannai--Ito polynomials defined on a discrete support  are interpreted as the overlap coefficients between the two  eigenbases in an irreducible finite dimensional vector space. With respect to the first basis, the generator $\kappa_1$ acts as a diagonal matrix with spectrum (\ref{spec-thet}) and the generator $\kappa_2$ acts as a tridiagonal matrix. With respect to the second basis, the generator $\kappa_2$ acts as a diagonal matrix with
Bannai--Ito's discrete spectrum (see \cite{GVZ}):
\beqa
 {\theta^*_n}'=(-1)^n(n/2 + h +1/4) -1/4, \quad n\in {\mathbb Z}\label{specstar}
\eeqa 
where $h$ is a scalar,  whereas  the generator $\kappa_1$ acts as a {\it five-diagonal} matrix.\vspace{1mm} 

\medskip

\subsubsection{Relation to the divided polynomials}
We are now in position to reconsider the example of \cite{GVZ} in light of the results presented in previous Sections. To this end, 
 we  choose $b$, $b^*$, $c$, and $c^*$ such that they have a pole at $q=e^{\rmi \pi/2}$ but the spectra of some of the operators among the corresponding cyclic TD pair and the divided polynomials remain finite  and non-trivial for $q=e^{\rmi \pi/2}$, and agree with the spectra for $\kappa_1$ and $\kappa_2$.  We fix
\beqa
a&=&0,\qquad b=\frac{\rmi}{4(q+q^{-1})^{1/2}}q^{2g+2},\quad \ c=-\frac{\rmi}{4(q+q^{-1})^{1/2}}q^{-2g-2},\label{paramBI}\\
a^*&=&-\frac{1}{4},\qquad b^*=-\frac{1}{4(q+q^{-1})}q^{4h+1}, \quad c^*=\frac{1}{4(q+q^{-1})}q^{-4h-1}\label{paramBI2}
\eeqa
where $g,h$ are arbitrary scalars. In this case, by straightforward calculations one finds that the spectrum of $A$ is trivial ($\theta_n=0$) and the spectrum of $\Aexs{2}$ is singular ($\tilde{\theta}_n^*\rightarrow \infty$). On the other hand, one finds that the spectra of $\Aex{2}$ and $A^*$  read, respectively:
\beqa
\tilde{\theta}_n=\frac{n^2}{4} +\frac{n}{2}(g+1),\qquad {\theta^*_n}=(-1)^n(n/2 + h+1/4) -1/4.\label{comp-BI}
\eeqa
We note that Lemma~\ref{lem3} can not be applied for our choice of $b$ and $c$, however the second derivative of $\theta_n(q)$ at $q=\rmi$ exists and we compute the spectrum following Remark~\ref{rem:Hopital}.
Observe that these expressions coincide with the spectra of the operators $\kappa_1$ in (\ref{spec-thet}), for the identification $\tilde{c}'_0=g/2+3/4$, and $\kappa_2$ in (\ref{specstar}), respectively, associated with the complementary Bannai--Ito polynomials defined on a discrete support~\cite{TVZ}.

\subsection{The dual $-1$ Hahn algebra}
As a limiting case of the complementary Bannai--Ito polynomials, for the dual $-1$ Hahn polynomials the algebraic structure that encodes all the properties of the polynomials is a degenerate version of the complementary Bannai--Ito algebra (\ref{CBIalg}) \cite{GVZ}. Using a certain limiting procedure, the defining relations instead read:
\beqa
\big[\kappa_1,r\big]=0,&& \{\kappa_2,r\}=2\gamma_3, \qquad r^2=\one,\label{dualHahn-alg}\\
\big[\kappa_1,\big[\kappa_1,\kappa_2\big]\big]&=&  -\gamma_2\big[\kappa_1,\kappa_2\big]r  + \gamma_1\kappa_2 -\gamma_1 \gamma_3 ,\nonumber\\
\big[\kappa_2,\big[\kappa_2,\kappa_1\big]\big]&=& \gamma_2\kappa_2^2r + 2\gamma_3\kappa_1r + 2\gamma_3\big[\kappa_1,\kappa_2\big]r +\kappa_1 + \gamma_4 r +\gamma_5.\nonumber
\eeqa
In this case,  the homomorphism:
\beqa
\kappa_1\rightarrow \cal{D}, \qquad \kappa_2 \rightarrow x
\eeqa
with (\ref{opD}) is considered.  The explicit expressions of the structure constants $\gamma_i,1,...,5$, in terms of $r_1,r_2,\rho_2$ is given in \cite[page 17]{GVZ}.  In particular:
\beqa
\gamma_1=\tilde{c}_0(1-\tilde{c}_0), \qquad \gamma_2=1-2\tilde{c}_0\qquad \mbox{and} \qquad  \gamma_3=\rho_2.\label{deltaHahn}
\eeqa
 By analogy with the complementary Bannai--Ito polynomials, the dual $-1$ Hahn polynomials defined on a discrete support  are interpreted as the overlap coefficients between the two eigenbases (for $\kappa_1$ and $\kappa_2$). Compared with the previous case, the spectrum of the  generator $\kappa_1$ is now given by (\ref{spec-nu}).\vspace{1mm}

\subsubsection{Relation to the divided polynomials}
For our second example, instead of (\ref{paramBI}) let us now choose:
\beqa
a=0,\qquad b=\frac{e^{\rmi \pi/4}}{2},\quad \ c=0 \label{paramHahn}
\eeqa
and (\ref{paramBI2}). In this case, the spectrum of $A$ is now given by $\theta_n=e^{\rmi\pi/4}(-1)^n/2$. On the other hand, the spectrum of $\Aex{2}$ simplifies to:
\beqa
\tilde{\theta}_n=n/2,\label{comp-Hahn}
\eeqa
whereas the spectrum of $A^*$ remains unchanged, given by $\theta^*_n$ in (\ref{comp-BI}). Observe that these expressions coincide  with the spectra of the operators $\kappa_1$ (\ref{spec-thet}) provided $\tilde{c}_0=1/2$ and $\kappa_2$ (\ref{specstar}), respectively, associated with the dual $-1$ Hahn polynomials defined on a discrete support.

\subsection{Non-symmetric generalization of the tridiagonal algebra} Recall that  the Askey--Wilson algebra (\ref{AWrel}) is a quotient of the tridiagonal algebra~\eqref{TD1}-\eqref{TD2}, i.e.,  the defining relations (\ref{AWrel}) imply  the defining relations~\eqref{TD1}-\eqref{TD2} but the reverse is not true in general.
 By analogy with this, we introduce a non-symmetric generalization of the tridiagonal algebra as  the associative $\K$-algebra with unit  generated by two elements $\Aex{2}$ and $A^*$ subject to the relations
\begin{align}
\Bigl[\Aex{2},(\Aex{2})^2 A^* -2\Aex{2} A^* \Aex{2} + A^*(\Aex{2})^2 -\gamma\big(\Aex{2} A^* +  A^*\Aex{2}\big) - \rho_0^{[2]} A^*\Bigr]&=0,\nonumber\\
\qquad\quad  \Bigl[A^*,{A^*}^4  \Aex{2} - 2{A^*}^2\Aex{2}{A^*}^2 + \Aex{2}{A^*}^4   + {A^*}^3  \Aex{2} - {A^*}^2\Aex{2}{A^*} - {A^*}\Aex{2}{A^*}^2 + \Aex{2}{A^*}^3 &\label{degrel}\\
 - {A^*}^2\Aex{2} - 2 A^*\Aex{2}A^* - \Aex{2}{A^*}^2 - A^*\Aex{2} -\Aex{2}A^* \Bigr]&=0,\nonumber
\end{align}
where $\gamma$ and $\rho_0^{[2]}$ are scalars. We show that the complementary Bannai--Ito~\eqref{CBIalg} and dual $-1$ Hahn~\eqref{dualHahn-alg} algebras are quotients of our  non-symmetric generalization of the tridiagonal algebra, for certain values of the scalars.

Indeed, for the choice (\ref{paramBI})-(\ref{paramBI2}) the relations satisfied by the two operators $\Aex{2}$ and $A^*$ can be derived based on the knowledge of the spectra and the action~\eqref{eq:t1-q} of $A^*$ on the eigenspaces for $\Aex{2}$. Applying similar techniques as the ones used to show Theorems  \ref{prop:mixed-rel-k-zero} and \ref{propNOns}, one finds that   the two operators satisfy the relations~\eqref{degrel}
with
\beqa
 \gamma=1/2 \quad \mbox{and} \quad  \rho_0^{[2]}=g^2/4+g/2+3/16.\nonumber
\eeqa
In the  case of the dual $-1$ Hahn algebra, the relations are satisfied with
\beqa
 \gamma=0 \quad \mbox{and} \quad  \rho_0^{[2]}=1/4.\nonumber
\eeqa

Now, observe  that the defining relations of   the complementary Bannai--Ito and dual $-1$ Hahn algebras imply\footnote{To obtain the first relation in~\eqref{degrel}, it is enough to compute the commutator of $\kappa_1$ with the left-hand side of~\eqref{CBIalg-2}, while for the second relation we use multiple commutators of~\eqref{CBIalg-3} with $\kappa_2$.} the relations~\eqref{degrel}: the second and third relations of (\ref{CBIalg}) with $\delta_2=0$ and $\delta_1=\rho_0^{[2]}$ (see (\ref{delta})) imply (\ref{degrel}) with $\gamma=1/2$.
Similarly, the second and third relations of (\ref{dualHahn-alg}) with $\gamma_2=0$ and $\gamma_1=\rho_0^{[2]}$ (see (\ref{deltaHahn})) imply (\ref{degrel}) with $\gamma=0$.\vspace{1mm}
We have thus demonstrated the following proposition.

\begin{Prop}\label{prop:BI} We have homomorphisms from the non-symmetric generalization of the tridiagonal algebra~\eqref{degrel} to the complementary Bannai--Ito and dual $-1$ Hahn algebras such as
\beqa
\Aex{2} \ \mapsto \kappa_1, \qquad A^*\mapsto \kappa_2,\label{mapkappa}
\eeqa
 with the kernel described by the relations for $\kappa_{1}$ and $\kappa_{2}$ given in~\eqref{CBIalg}, for the first case, and in~\eqref{dualHahn-alg} for the second.
\end{Prop}

\begin{rem}
It is important to stress that the relations (\ref{degrel}) hold in general for any dimension of the eigenspaces $V_n$ and $V^*_s$ -- i.e., not necessarely one-dimensional eigenspaces -- contrary to the relations (\ref{CBIalg}) or (\ref{dualHahn-alg}).  
\end{rem}

\begin{rem}
According to the results of the previous Sections, it implies that the operator $\kappa_1$ can be written as a limit for $q=e^{\rmi \pi/2}$ of a quadratic combination of a second-order $q-$difference operator with a non-degenerate spectrum of $q-$Racah type. Indeed, it is easy to check that the structure of $\Aex{2}$ with (\ref{WNnew}) for (\ref{paramBI}) or (\ref{paramHahn}) and (\ref{paramBI2}) agrees  with the proposals \cite[eq. (4.2)]{GVZ} and \cite[eq. (5.6)]{TVZ}, respectively.
\end{rem}


To conclude, recall that tridiagonal pairs are certain generalizations of Leonard pairs. In this picture, the multivariable Gasper-Rahman orthogonal polynomials \cite{GR1} find a natural interpretation within the theory of tridiagonal pairs \cite{BM}, thanks to the construction of Iliev \cite{Iliev}. Thus, multivariable extensions of the complementary Bannai--Ito or dual $-1$ Hahn polynomials can be easily constructed  by using the algebra~\eqref{degrel}. It is clear that a potential starting point for the analysis of the corresponding extended bispectral problem is given by the elements (\ref{WNnew}) for $N=2$. 
\vspace{2mm}

\noindent{\bf Acknowledgments:} We thank P. Terwilliger for important comments on the manuscript and H. Saleur for discussions, comments, as well as interest in this work. P.B thanks  A.  Grunbaum, E. Koelink, L. Vinet and A. Zhedanov for pointing out some of the references and communications. P.B thanks CRM Montr\'eal/UMI 3457 for invitation and support, where part of this work has been completed.  A.M.G. thanks LMPT Tours for its hospitality during the visit in 2015 when this work was started.
P.B. and A.M.G. are supported by C.N.R.S.   The work of A.M.G. was also supported by a Humboldt fellowship and DESY. 
\vspace{0.5cm}

\begin{appendix}
\newcommand{\ccc}{\alpha}
\section{Examples of cyclic TD pairs}\label{app:example1}

In this Appendix, we construct  our two basic examples of cyclic TD pairs. They both come from representation theory of quantum algebras specialized  at roots of unity.

\subsection{Cyclic representations}\label{app:first-ex}
Let  $q=e^{\rmi\pi/4}$ and  $A_{\pm}$ 
 denote the $4\times 4$ matrices 
\begin{equation}\label{exL4:W1}
A_{\pm}\qquad = \qquad
\left(
\begin{array}{cccc}
 0 & q^{\pm1} k_+ & 0 & q^{\mp2} \ccc^2 k_- \\
 q^{\pm1} k_- & 0 & 2 k_+ & 0 \\
 0 & k_- & 0 & q^{\mp1} k_+ \\
 0 & 0 & q^{\mp1} k_- & 0 \\
\end{array}
\right)
\end{equation}
where $k_{\pm}$ and $\ccc$ are complex numbers.
These matrices are diagonalizable with the eigenvalues
\begin{equation}\label{specA1}
\mathrm{Spec} \; A_{\pm}:\qquad
\left\{\pm\sqrt{q^{\pm3} \ccc k_-^2+k_+k_-},\pm\sqrt{q^{\mp1}\ccc k_-^2 + k_+k_-}\right\}.
\end{equation}
We show that $A_\pm$ form a cyclic  TD pair  of cyclicity $4$ under certain conditions on $\ccc$ and $k_\pm$. Assume first that $k_-$ and $\alpha$ are non-zero,
 and moreover $\alpha\ne \pm q k_+/k_-$. The 
eigenstates of $A_+$ can be written in components as
\begin{align*}
v_0&=\left\{q^3 c_-\ffrac{\sqrt{k_- \left(q^{-1}\ccc k_-+k_+\right)} }{k_-},
\ccc+q^2\ffrac{\sqrt{2} k_+}{k_-},-q\ffrac{\sqrt{k_- \left(q^{-1} \ccc k_-+k_+\right)}}{k_-},1\right\},\\
v_1&=\left\{q^3 c_+\ffrac{ \sqrt{k_- \left(q^3\ccc k_-+k_+\right)}}{k_-},
-\ccc+q^2\ffrac{\sqrt{2} k_+}{k_-},q\ffrac{ \sqrt{k_- \left(q^3\ccc k_-+k_+\right)}}{k_-},1\right\},\\
v_2&=\left\{-q^3 c_-\ffrac{\sqrt{k_- \left(q^{-1}\ccc k_-+k_+\right)} }{k_-},
\ccc+q^2\ffrac{\sqrt{2} k_+}{k_-},q\ffrac{ \sqrt{k_- \left(q^{-1}\ccc k_-+k_+\right)}}{k_-},1\right\},\\
v_3&=\left\{-q^3 c_+\ffrac{\sqrt{k_- \left(q^3\ccc k_-+k_+\right)}}{k_-},
-\ccc+q^2\ffrac{\sqrt{2} k_+}{k_-},-q\ffrac{ \sqrt{k_- \left(q^3\ccc k_-+k_+\right)}}{k_-},1\right\},
\end{align*}
where we set $c_{\pm} =  (\ccc k_-\pm\sqrt{2}k_+)/k_-$. 
Note that these vectors are indeed linearly independent, with our assumption on $\alpha$ and $k_-$.
We observe that the action of $A_-$ on any of these eigenvectors~$v_p$, with $p\in\mathbb{Z}_4$, is tridiagonal with non-zero coefficients and corresponds to the action of $C^*=A_-$ in~\eqref{eq:ct1} if we set $V^{(N=4)}_p=\mathbb{C} v_p$. 
For example, the coefficients in the action of $A_-$ on $v_0$ (i.e. $A_- v_0 = c_0 v_0+c_1 v_1 +c_2 v_2 +c_3 v_3 $) are
$c_2=0$, 
$c_0=-2 q\ffrac{ k_+ \sqrt{k_-(q^{-1}\ccc k_-+k_+)}}{\ccc k_-}$ (recall
 our assumption on $\ccc$ and $k_-$ above), and $c_1$ and $c_3$ are also non-zero (but more complicated) expressions in $k_{\pm}$ and~$\ccc$.  Similarly, the action of $A_-$ on $v_1$ is a linear combination of $v_0$, $v_1$ and $v_2$, etc. We similarly have an eigenbasis for $A_-$ and the cyclic tridiagonal action of $A_+$. This shows that the conditions (i)-(iii) in Definition~\ref{deficitri} are satisfied for $N=4$,  if the spectrum~\eqref{specA1} is non-degenerate. Let us call  values of $\alpha$ and $k_{\pm}$ \textit{generic} if the  spectrum~\eqref{specA1}  is non-degenerate and the coefficients $c_{i,i}$ and $c_{i,i\pm1}$ in the action $A_- v_i = \sum_{j=0}^3 c_{i,j} v_j$ are non-zero and similarly for $c^*_{i,j}$ for the action of $A_+$ in the second basis. Then at such  generic values of  $\alpha$ and $k_{\pm}$  our example provides a family of cyclic TD pairs which are not ordinary TD pairs.
\vspace{1mm}

 In this example, it is easy to see that for generic $\alpha$ and $k_{\pm}$ the condition (iv) holds as well or that the action of $A_{\pm}$ is irreducible. Indeed, we first note that the  spectrum of $A_+$ is non-degenerate, therefore an invariant subspace $W$ is a span of the (one-dimensional) $A_+$-eigenspaces (otherwise, $W$ would not be invariant even with respect to the action of $A_+$). 
Then, we do not have a  
1-dim invariant subspace: assume it is spanned by the $A_+$-eigenvector $v_i$ for some  
$i=0,1,2,3$, then its image $A_- v_i$ is in $\oC v_{i-1} + \oC v_i +   
\oC v_{i+1}$, which is not in $\oC v_i$. Suppose that we have a  
2-dimensional invariant subspace $W=\oC v_i + \oC v_j$, then consider  
$v_i$ -- its image under the action of $A_-$ is $c_{i,i}v_i +  
c_{i,i+1}v_{i+1} + c_{i,i-1}v_{i-1}$ with non zero coefficients, and  
it is thus not in $W$ for any choice of $v_j$. Suppose then we have a  
3-dimensional invariant subspace $W=\oC v_i + \oC v_j + \oC v_k$, then  
again  consider $v_i$ -- its image under the action of $A_-$ is  
$c_{i,i}v_i + c_{i,i+1}v_{i+1} + c_{i,i-1}v_{i-1}$ with non zero  
coefficients. This action of $A_-$ on $v_i \in W$ forces us to fix  
$j=i+1$ and $k=i-1$ or  $j=i-1$ and $k=i+1$. For any of these choices,  
we have that $v_{i+1}\in W$ while $A_- v_{i+1} = c_{i+1,i+1}v_{i+1} +  
c_{i+1,i+2}v_{i+2} + c_{i+1,i}v_{i}$, again with non-zero coefficients  
for generic parameters $\alpha$ and $k_{\pm}$. We see here the  
non-zero contribution of $\oC v_{i+2}$ which is not in $W$. We  
therefore conclude that there is only non-zero 4-dimensional invariant  
subspace $W=V$. Thus, there is no invariant proper subspaces
 under the action of the pair $A_{\pm}$, \textit{i.e.}, the condition (iv) of Definition~\ref{deficitri} holds indeed. We thus have the action of a cyclic TD pair with cyclicity~$4$. 
  Note that we have actually a cyclic Leonard pair here (the cyclic TD pair with `tridiagonal' instead of `block tridiagonal' action).
 \vspace{1mm}

We note that 
setting the parameter $\alpha=0$ in (\ref{exL4:W1}), the  pair $A_{\pm}$ becomes the (usual) TD pair of $q-$Racah type at $q=e^{\rmi\pi/4}$, i.e. with the action from Definition~\ref{tdp}. We have thus a continuous family of cyclic TD pairs parametrized by $\alpha$ that has a special point at $\alpha=0$.
We also note that for non-zero $\alpha$ our example  is not of $q-$Racah type, i.e. it is not a specialization of a TD pair of $q-$Racah type.
 \vspace{1mm}
 
We give a final remark on the origin of the construction of the $A_{\pm}$ operators. It is the result of the pull-back representation of the $q-$Onsager algebra  (identifying $A_+=\cW_0$ and $A_-=\cW_1$ in our $q-$Onsager algebra notations) under the homomorphism \cite{Bas3}
\beqa
\qO\to U_q (sl_2): \qquad A_{\pm} \mapsto k_+ q^{\pm1/2} E \sqrt{K^{\pm1}} + k_- q^{\mp1/2} F \sqrt{K^{\pm1}} + \epsilon_{\pm} K^{\pm1}\label{mapOqsl2Uqsl2}
\eeqa
 in  the $4$ dimensional semi-cyclic representation of $U_q (sl_2)$ at $q=e^{\rmi\pi/4}$. 
Recall that the $U_q (sl_2)$ action is given in a basis $\{w_j, \,  0\leq j\leq 3 \}$ by (see e.g.~\cite[Sec. 11.1]{CP})
\begin{equation}
K w_j = q^{3-2j} w_j, \qquad
E w_j =  [j]^2w_{j-1},\qquad
 F w_j = w_{j+1},\qquad \text{with} \quad w_4=\alpha^2 w_0.\label{mapKEF}
\end{equation}
The two matrices $A_\pm$ given in (\ref{exL4:W1}) are derived from (\ref{mapOqsl2Uqsl2}) with (\ref{mapKEF}) where, for simplicity,
we have chosen $\epsilon_\pm=0$. By construction \cite{Bas3}, $A_\pm$ satisfy the defining relations of the $q-$Onsager
 algebra (2.11) with $\rho=(q+q^{-1})^2k_+ k_-$, as can be checked explicitly.

\newcommand{\kp}{k_+}
\newcommand{\km}{k_-}

\subsection{$q-$Racah type representations}\label{app:q-Racah}
We consider now a $q-$Racah type representation of the $q-$Onsager algebra and show that it  satisfies Assumption~\ref{ass}. We thus construct a cyclic TD pair of the $q-$Racah type. The simplest and non-trivial one we were able to find is on an 8-dimensional vector space. We construct it as follows. Recall  the homomorphism of the $q-$Onsager algebra $\qO$ to the quantum affine algebra $U_q (\widehat{sl_2})$ defined on generators as
 \cite{Bas3}
\begin{equation}\label{q-O-hom-affine}
\begin{split}
\cW_0 \mapsto (k_+ E_1 + k_-  F_1) \sqrt{K_1} + \epsilon_{+} K_1,\\
\cW_1 \mapsto (k_+  F_0 + k_- E_0) \sqrt{K_0} + \epsilon_{-} K_0,
\end{split}
\end{equation}
where $E_i$, $F_i$, and\footnote{Note that we use here $q$ instead of $q^{1/2}$ in \cite{Bas3}.} $\sqrt{K_i}=q^{H_i/2}$ are the standard Chevalley generators of $U_q (\widehat{sl_2})$. 
The image of this homomorphism is a coideal subalgebra and one can use the iterated coproduct of $U_q (\widehat{sl_2})$ to define an evaluation representation of $\qO$ on the 8-dimensional  vector space (for any $u_i\in\mathbb{C}^{\times}$)
\begin{equation}\label{ex2-V}
V=\mathbb{C}^2(u_1)\otimes\mathbb{C}^2(u_2)\otimes\mathbb{C}^2(u_3) \ .
\end{equation}
This representation is obtained by the recursive formula (applied for  $n=3$ here) \cite{Bas3}
\begin{equation}\label{q-O-XXZ-action}
\begin{split}
\cW_0^{(n)} &= \bigl(\kp u_n  q^{1/2}  \sigma_+ + k_- u_n^{-1} q^{-1/2} \sigma_- \bigl)\sqrt{K} \otimes \one^{(n-1)} + K\otimes  \cW_0^{(n-1)}  ,\\
\cW_1^{(n)} &=\bigl(\kp u_n^{-1}  q^{-1/2}  \sigma_+ + k_- u_n q^{1/2} \sigma_- \bigl)\sqrt{K^{-1}} \otimes \one^{(n-1)} + K^{-1}\otimes  \cW_1^{(n-1)},
\end{split}
\end{equation}
with $\cW_0^{(0)}=\epsilon_+$ and $\cW_1^{(0)}=\epsilon_-$, $\sigma_{\pm}$ are Pauli matrices\footnote{$\sigma_\pm=(\sigma_x\pm \mbox{i}\sigma_y)/2$ and $\sigma_{x,y,z}$ are the usual Pauli matrices
\beqa
\sigma_x=\left(
\begin{array}{cc}
 0    & 1 \\
 1 & 0 
\end{array} \right) \ ,\qquad
\sigma_y =\left(
\begin{array}{cc}
 0    & -\mbox{i} \\
\mbox{i} & 0 
\end{array} \right) \ ,\qquad
\sigma_z =\left(
\begin{array}{cc}
 1    & 0 \\
 0 & -1 
\end{array} \right) \ .\label{Pauli}\nonumber
\eeqa
} and $\sqrt{K}$ is the diagonal $2\times 2$ matrix with entries $q^{1/2}$ and $q^{-1/2}$.

We fix the three evaluation parameters for simplicity as $u_1=u_2=u_3=1$, abbreviate $\cW_i=\cW_i^{(3)}$ and  set $\epsilon_{\pm}=k_{\pm}=1$ to make the exposition below clearer.
In a basis in $V$, we have then the matrices
\begin{equation}\label{q-O-XXZ-action-1}
\cW_{0(1)}\; \equiv\; A_{\pm}\; =\; \left(
\begin{array}{cccccccc}
 q^{\pm3} & q^{\pm2}   & q^{\pm1}    & 0 &  1 & 0 & 0 & 0 \\
 q^{\pm2}  & q^{\pm1} & 0 & q^{\pm1}    & 0 &  1 & 0 & 0 \\
 q^{\pm1}   & 0 & q^{\pm1} & 1  & 0 & 0 & 1  & 0 \\
 0 & q^{\pm1}   & 1 & q^{\mp1} & 0 & 0 & 0 & 1  \\
 1 & 0 & 0 & 0 & q^{\pm1} & 1  & q^{\mp1}    & 0 \\
 0 & 1 & 0 & 0 & 1 & q^{\mp1} & 0 &  q^{\mp1}   \\
 0 & 0 & 1 & 0 & q^{\mp1}   & 0 & q^{\mp1} & q^{\mp2}   \\
 0 & 0 & 0 & 1 & 0 & q^{\mp1}  & q^{\mp2}   & q^{\mp3} \\
\end{array}
\right)\ ,
\end{equation}
where we set $A_+$ for $\cW_0$ and $A_-$ for $\cW_1$.
With respect to the $\cW_0$ action
thus defined, the vector space $V=\bigl(\mathbb{C}^2\bigr)^{\otimes 3}$ 
 is decomposed onto a direct sum of four $\cW_0$-eigenspaces
\beqa\label{ex2-V-decomp}
V=V_{[-2]} \oplus V_{[0]}\oplus V_{[2]} \oplus V_{[4]},
\eeqa
where each eigenspace $V_{[\lambda]}$ corresponds to the eigenvalue  $[\lambda]_q$, and $V_{[0]}$ and $V_{[2]}$ are $3$-dimensional 
while $V_{[-2]}$ and $V_{[4]}$ are 1-dimensional eigenspaces.
 Note that the spectrum of $\cW_0$ is of the $q-$Racah type, i.e. of the form (\ref{specgen-I}) with the identification:
\beqa\label{eq:abc-ex2}
a=0,\qquad b= \frac{q^{-2}}{q-q^{-1}}, \qquad c= -\frac{q^{2}}{q-q^{-1}}\label{valabc}
\eeqa
and $p=0,1,2,3$.
We denote the corresponding  $\cW_0$-eigenvectors for~$V_{[\lambda]}$ as $v_k^{[\lambda]}$, where $k\in\{1,2,3\}$ for $\lambda$ equal $0$ or $2$, and $k=1$ for $\lambda$ equal $-2$ or $4$. Then, for $\lambda=-2$ we found
\begin{align}\label{eq:eigenv1}
v_1^{[-2]} &= \left(-1,q,1,-q,q^{-1},-1,-q^{-1},1\right)\qquad\qquad\qquad\;
\end{align}
and  for $\lambda=0$
\begin{align}
v_1^{[0]} & = \left(q^{-2},q^{-3}-q^{-1},q^{-4}-q^{-2},-q^{-3},-q^{-3},0,0,1\right) , \label{eq:eigenv2}\\ 
v_2^{[0]} & = \left(0,q^{-2},q^{-3}-q^{-1},-q^{-2},-q^{-2},0,1,0\right) , \label{eq:eigenv3}\\
v_3^{[0]} & = \left(0,0,q^{-2},-q^{-1},-q^{-1},1,0,0\right) , \label{eq:eigenv4}
\end{align}
and  for $\lambda=2$
\begin{align}
v_1^{[2]} & = \left(0,0,-1,-q^{-1},q,1,0,0\right), \label{eq:eigenv5}\\
v_2^{[2]} & = \left(0,-1,q-q^{-1},-q^{-2},1,0,1,0\right), \label{eq:eigenv6}\\
v_3^{[2]} & = \left(-1,q-q^{-1},1-q^{-2},q+q^{-1}-q^{-3},-q,0,0,1\right),\label{eq:eigenv7}
\end{align}
and for $\lambda =4$
\begin{equation}\label{eq:eigenv8}
 v_1^{[4]} = q^3\left(q^3,q^2,q^1,1,1,q^{-1},q^{-2},q^{-3}\right).\qquad\qquad\qquad\qquad\quad\;
\end{equation}
%
We next study the action of $\cW_1$ on these 8 vectors (a direct calculation using~\eqref{q-O-XXZ-action-1}) and  check that  $\cW_1$ acts as a block tridiagonal matrix: the standard ordering of $\cW_0$-eigenspaces here is 
\begin{equation}\label{ex2-ord}
V_0= V_{[-2]} , \qquad V_1 =V_{[0]},\qquad V_2 =V_{[2]},\qquad V_3 =V_{[4]}\ ,
\end{equation}
i.e., the action is 
\begin{equation}\label{app:ex2-W1}
\cW_1 V_0 \subseteq V_0 \oplus V_1, \qquad \cW_1 V_1 \subseteq V_0 \oplus V_1 \oplus V_2,
\qquad \cW_1 V_2 \subseteq V_1 \oplus V_2 \oplus V_3,
\qquad \cW_1 V_3 \subseteq V_2 \oplus V_3,
\end{equation}
 and the coefficients in the action are certain rational functions in $q$ where the denominator is a power of $q$ or $(1+q^{2n})$, for $n\in\{1,2,3\}$, or $(1-q^2+q^4)$, or the product of them.
We have also a  decomposition onto $\cW_1$-eigenspaces as in~\eqref{ex2-V-decomp} (i.e., $\cW_1$ has the same spectrum),  and similarly we construct the $\cW_1$-eigenvectors as in~\eqref{eq:eigenv1}-\eqref{eq:eigenv8} by replacing $q\to q^{-1}$.  In this basis, the action of $\cW_0$ is block tridiagonal where the coefficients are again rational functions in $q$ with the denominators as described just above. 

We have computed  the transition matrix $S$ from the $\cW_0$-eigenbasis to the $\cW_1$-eigenbasis explicitly  (though we do not give it here, as its entries are rather complicated) and notice that all its entries are non-zero rational functions in $q$ with the common denominator 
\beq\label{eq:t}
t=q^5(1-q^2+q^4)(1+q^{4})(1+q^{2})^2\ .
\eeq 

 \vspace{1mm}
 
We prove next that this representation of $\qO$ is irreducible\footnote{We can not use the result~\cite[Theorem 1.17]{ITaug} on irreducibility of certain tensor-product representations  of $\qO$, as we consider different embedding of $\qO$ into the quantum affine algebra $U_q(\widehat{sl}_2)$ and thus we deal with different  tensor-product  representations of $\qO$.}.

\begin{prop}\label{prop:V-irrep}
The representation of $\qO$ on $V$ given by~\eqref{q-O-XXZ-action-1} is irreducible.
\end{prop}

We begin our proof with the very simple lemma.
\begin{sublem}\label{lem:inv-subspace-1}
 Any  subspace $W\subset V$ invariant under the action of $\qO$ is a direct sum of $\cW_0$-eigenspaces.
 \end{sublem}
\begin{proof}
Assume that $W\subset V$ is invariant under the action of $\qO$  and therefore it is invariant under the action of  the commutative algebra $\oC\cW_0$. We have shown that $\cW_0$ is diagonalisable and thus $V$ is a direct sum of $1$-dimensional irreducible representations of $\oC\cW_0$ (we have $4$ isomorphism classes of them).  Therefore, an invariant subspace $W\subset V$ is a direct sum of the irreducible $\oC\cW_0$ representations and this is the statement of the lemma.
\end{proof}

We then study the action of $\cW_1$ on the  $\cW_0$-eigenvectors of the highest- and lowest-weights $[4]_q$ and $[-2]_q$   correspondingly.
\begin{sublem}\label{lem:inv-subspace-2}
Assume that an invariant subspace $W\subseteq V$ contains one of the $\cW_0$-eigenvectors $v_1^{[-2]}$ or $v_1^{[4]}$ of the eigenvalues  $[-2]_q$ and $[4]_q$, correspondingly,  then $W=V$.
\end{sublem}
\begin{proof}
We begin with $v_1^{[-2]}$ and prove the statement by a direct and recursive calculation of the action of $\cW_1$ and $\cW_0$. We construct explicitly a basis in $V$ by applying the following words in the generators
\begin{equation}\label{eq:words-basis}
\cW_1, \quad \cW^2_1, \quad \cW^3_1, \quad \cW_0 \cW^2_1, \quad \cW_1 \cW_0 \cW^2_1, \quad
\cW_0 \cW_1 \cW_0 \cW^2_1, \quad \cW_1 \cW_0 \cW_1 \cW_0 \cW^2_1
\end{equation}
on $v_1^{[-2]}$ and it is a straightforward check that the eight vectors (together with $v_1^{[-2]}$) are indeed  linearly independent. Therefore, if $v_1^{[-2]}$ is in a submodule it generates the whole module $V$. We do the similar calculation for $v_1^{[4]}$, using the same words in the $q-$Onsager algebra generators.
\end{proof}

\begin{proof}[Proof of Prop.~\ref{prop:V-irrep}]
By  Lemma~\ref{lem:inv-subspace-2} we thus have that a proper invariant subspace $W$ (if it exists) is a subspace in $V_1\oplus V_2$. 
%
We then analyze the action of $\cW_1$ and $\cW_0$ on the direct sum $V_1\oplus V_2$. Assume that there exists a proper submodule $W$ in $V_1\oplus V_2$. A vector generating this submodule is in $V_1$ or $V_2$, due to Lem.~\ref{lem:inv-subspace-1}. Assume it is in $V_1$ and denote it by $w_1$.  It can be written as a linear combination of the basis vectors in  $V_1=V_{[0]}$, see~\eqref{eq:eigenv2}-\eqref{eq:eigenv4},
\begin{equation}
w_1 = \sum_{k=1}^3c_k v_{k}^{[0]}
\end{equation}
where $c_k\in\oC(q)$ are subject to certain conditions. The idea is that acting by $\cW_1$ on $w_1$ we produce   $v_{1}^{[-2]}$ or $v_{1}^{[4]}$ and the corresponding coefficients, as linear functions of $c_k$, have to be  $0$, otherwise $W=V$ due to Lem.~\ref{lem:inv-subspace-2}. It turns out  that it is enough to  apply $\cW_1$ and $\cW_1^2$ and solve the corresponding linear equations on the coefficients in front of $v_{1}^{[-2]}$ and $v_{1}^{[4]}$.  This way we recursively find that all $c_1=c_2=c_3=0$ 
and thus a contradiction to the assumption on $W$. 
We do the similar analysis for $V_2=V_{[2]}$ and arrive at the same conclusion that there exists no vector generating a proper submodule in $V_1\oplus V_2$, and therefore  there are no proper invariant subspaces for the action of $\cW_0$ and $\cW_1$ in $V$, or
 the module $V$ is irreducible.
We have thus shown explicitly  that $\cW_0$ and $\cW_1$ form a TD pair of the  $q-$Racah type.
\end{proof}

\subsubsection{Specialization to $q=e^{\rmi\pi/3}$}
 We specialize now $q$ to the root of unity $e^{\rmi\pi/3}$ and consider the representation of $\cW_i=\cW_i^{(3)}$ defined by~\eqref{q-O-XXZ-action-1} at this value of~$q$.  We notice that the points (i)-(iii) of Assumption~\ref{ass} hold: for the point (i), the common denominator of the diagonal and block tridiagonal matrices for $\cW_0$ and $\cW_1$ is some power of $q$ multiplied by $(1-q^2+q^4)\prod_{n=1}^3(1+q^{2n})$ that has a finite non-zero value at  $q=e^{\rmi\pi/3}$; for the point (ii),  the entries of the transition matrix $S$ from  the $\cW_0$-eigenbasis to the $\cW_1$-eigenbasis  have  also denominators with finite non-zero values at the specialization, recall the common denominator $t$ of the entries of $S$ in~\eqref{eq:t}; and  for the point (iii) the ratio $c/b=q$ (recall that $q^{\pm3}=-1$).
The vector space $V=\bigl(\mathbb{C}^2\bigr)^{\otimes 3}$ is then decomposed onto a direct sum of three $\cW_0$-eigenspaces  (in agreement with Prop.~\ref{lem1})
$$
V=  V_{[-1]}\oplus V_{[0]} \oplus V_{[+1]} \ ,
$$
where $3$-dimensional eigenspaces $V_{[1]}$ and $V_{[0]}$ correspond to the eigenvalues $[1]_q=1$ and $0$, respectively, and $V_{[-1]}$ is to the  $(-1)$-eigenvalue, it is two-dimensional.
Note that the two 1-dimensional eigenspaces $V_0(q)\equiv V_0$ and $V_3(q)\equiv V_3$ from~\eqref{ex2-ord} at generic $q$ have a well defined  specialization at $q= e^{\rmi\pi/3}$ (in the basis~\eqref{eq:eigenv1} and~\eqref{eq:eigenv8} stated above) and contribute to the same $(-1)$-eigenvalue eigenspace  $V_{[-1]}$, and $v_1^{[-2]}$ and $v_1^{[4]}$ form an eigenbasis for $V_{[-1]}$. We will denote these basis elements as $v_1^{[-1]}$ and $v_2^{[-1]}$, correspondingly.
The  $\cW_0$-eigenvectors in $V_{[0]}$ and $V_{[1]}$ are the specialization (the limit $q\to e^{\rmi\pi/3}$) of the vectors in~\eqref{eq:eigenv2}-\eqref{eq:eigenv4} and~\eqref{eq:eigenv5}-\eqref{eq:eigenv7}, correspondingly, and they form a basis in $V_{[0]}$ and $V_{[1]}$.
It is then straightforward to check that the action of $\cW_1$ (under the specialization of~\eqref{q-O-XXZ-action-1}) in this basis is cyclic:
\begin{equation}\label{eq:W1-act-cyclic-ex2}
\begin{split}
\cW_1\, V_{[\pm1]} &\subset V_{[-1]} \oplus V_{[0]} \oplus V_{[+1]}\ ,\\ 
\cW_1\, V_{[0]} &\subset V_{[-1]} \oplus V_{[0]} \oplus V_{[+1]}\ ,
\end{split}
\end{equation}
with non-trivial contribution of each of the three direct summands in the action. 
The eigenvectors for $\cW_1$ are constructed by setting $q$ to $e^{-\rmi\pi/3}$ in~\eqref{eq:eigenv1}-\eqref{eq:eigenv8}.  The action of $\cW_0$ in this basis is then also  cyclic and tridiagonal as in~\eqref{eq:W1-act-cyclic-ex2}.
\vspace{1mm}

We now prove 
that  $V$ does not have proper  subspaces invariant under the action of $\cW_0$ and $\cW_1$. The idea of the proof is similar to the case of TD pairs above -- the only difference is that we do not have one-dimensional eigenspaces and our arguments are thus slightly modified. In particular, Lemma~\ref{lem:inv-subspace-2} is modified as follows.
\begin{sublem}\label{lem:inv-subspaceR-2}
Assume that an invariant subspace $W\subseteq V$ contains a non-zero $\cW_0$-eigenvector  of the eigenvalue  $-1$  then $W=V$.
\end{sublem}
\begin{proof}
Recall that $V_{[-1]}$ is of dimension two and we have the basis elements  $v_1^{[-1]}$ and $v_2^{[-1]}$ in $V_{[-1]}$ given by the specialization of~\eqref{eq:eigenv1} and~\eqref{eq:eigenv8}, correspondingly.
We first assume that $v_1^{[-1]}\in W$ and prove that then $W=V$ by a direct calculation of the action of $\cW_1$ and $\cW_0$. We construct explicitly a basis in $V$ by applying the following words in the generators (note a difference from~\eqref{eq:words-basis})
$$
\cW_1, \quad \cW^2_1, \quad \cW_0 \cW^2_1, \quad \cW_1 \cW_0 \cW^2_1, \quad
 \cW_1^2 \cW_0 \cW^2_1, \quad \cW_0 \cW_1^2 \cW_0 \cW^2_1, \quad \cW_1 \cW_0 \cW_1^2 \cW_0 \cW^2_1
$$ 
on $v_1^{[-1]}$ and it is a straightforward check that the eight vectors (together with $v_1^{[-1]}$) are indeed  linearly independent. Therefore, if $v_1^{[-1]}$ is in a submodule it generates the whole module $V$. We do the similar calculation for $v_2^{[-1]}$, using the same words in  $\cW_1$ and $\cW_0$, and check that $v_2^{[-1]}$ generates $V$. Therefore, if there exists a $\cW_0$-eigenvector (of eigenvalue $-1$) generating a proper invariant subspace it can be chosen as the linear combination $v = v_1^{[-1]} + c v_2^{[-1]}$, for $c\in\oC^{\times}$. Applying then $\cW_1$ on this combination we get  
$\cW_1 v = -3v_1^{[-1]} - 4v + w$,
where $w\in V_{[0]}\oplus V_{[1]}$. Therefore for any $c\in\oC^{\times}$, the action generates $v_1^{[-1]}$ and thus the whole module $V$, due to the previous steps.
\end{proof}

By this lemma we thus have that a proper invariant subspace $W$ (if it exists) is a subspace in $V_{[0]}\oplus V_{[1]}$. The rest of the proof literally repeats the proof of Proposition~\ref{prop:V-irrep} given after the proof of Lemma~\ref{lem:inv-subspace-2}.
We have thus demonstrated that   all the conditions in Definition~\ref{deficitri} hold  for $N=3$, and the pair of operators given by the matrices in~\eqref{q-O-XXZ-action-1} at $q=e^{\rmi\pi/3}$  indeed forms a cyclic TD pair.

\subsection{Divided polynomials}\label{app:ex-2-div}
 In this section, we compute the divided polynomials~\eqref{WNnew} associated with (\ref{q-O-XXZ-action}) 
at $N=3$ or $q=e^{\rmi\pi/3}$. To this end, one needs first to compute powers $\cW_i^m$, $i=0,1$ for $m=2,3$. Define  
$w=k_+ \sigma_+ + k_-\sigma_-$. 
From (\ref{WNnew}) with (\ref{polymin}) and (\ref{valabc}), by straightforward calculation one finds:  
\begin{eqnarray*}
\mathcal{W}_0^{\footnotesize[3]}&=&  w\otimes w \otimes w + \frac{k_+k_-}{[2]_q}\left(w \otimes \mathbb{I} \otimes  \mathbb{I} + w\otimes  q^{2\sigma_z} \otimes \mathbb{I} - \mathbb{I}\otimes w \otimes \mathbb{I} \right)
\\&+&\epsilon_+ \left( w \otimes w \otimes q^{\sigma_z}  + w\otimes q^{2\sigma_z} \otimes w - \mathbb{I}\otimes w \otimes w \right)\\
 &+&\frac{\epsilon_+^2}{[2]_q}  \left( w \otimes q^{2\sigma_z}   \otimes q^{2\sigma_z}  - \mathbb{I} \otimes w \otimes q^{2\sigma_z}  + \mathbb{I}\otimes \mathbb{I} \otimes w \right),
\end{eqnarray*}
and similarly for $\cW_1^{\footnotesize[3]}$, replacing $\epsilon_+\rightarrow \epsilon_-$ and $q\rightarrow q^{-1}$ in the above expression. 
Note that the action of the divided polynomials on $V_p$, $p=0,1,2,3$ is given by Lemma \ref{lem2} and Proposition \ref{actWN} for $N=3$.\vspace{1mm}

 To conclude this section, let us point out that for the two operators defined in (\ref{q-O-XXZ-action}) and generic values of $q$,  their  spectra take the form  (\ref{specgen-I}). For the first operator, in terms of the parameters $\epsilon_+,k_\pm$ the identification is:
\beqa
&& a=0,\quad b= \frac{1}{2}\left(\epsilon_+-\sqrt{\epsilon_+^2+\frac{4k_+k_-}{(q-q^{-1})^2}}\right)q^{-3}\ ,\quad 
c=\frac{1}{2}\left(\epsilon_++\sqrt{\epsilon_+^2+\frac{4k_+k_-}{(q-q^{-1})^2}}\right)q^{3}.
\eeqa
Setting $\epsilon_\pm=-(q-q^{-1})$ and $k_{\pm}$ such that $\sqrt{k_+ k_-}=q-q^{-1}$, one observes that the coefficients  $a$, $b$, and $c$ (together with  $a^*$, $b^*$, and $c^*$) satisfy the conditions~\eqref{eq:conditions-bc}, with $\beta=-2$.  Therefore, the four operators $\W_i$ and $\cW_i^{\footnotesize{[3]}}$, for $i=0,1$, satisfy the relations in Theorem~\ref{prop:mixed-rel-k-zero} and Theorem~\ref{propNOns}.

\bigskip

\end{appendix}

\end{document}